\documentclass[12pt]{amsart}
\usepackage{amsmath,amssymb,amsthm,tikz,tikz-cd,color,xcolor,mathdots,amscd}
\usepackage[centertableaux]{ytableau}
\usepackage[all]{xy}
\usepackage[margin=1in]{geometry}
\usepackage{placeins}
\usetikzlibrary{shapes,arrows,svg.path}
\usepackage{booktabs}

\usepackage[colorlinks=true, pdfstartview=FitV, linkcolor=blue, citecolor=blue, urlcolor=blue]{hyperref}	

\theoremstyle{plain}
\newtheorem{theorem}{Theorem}[section]
\newtheorem{proposition}[theorem]{Proposition}

\newtheorem{lemma}[theorem]{Lemma}
\newtheorem{conjecture}[theorem]{Conjecture}

\theoremstyle{definition}

\theoremstyle{remark}
\newtheorem{remark}[theorem]{Remark}
\newtheorem{example}[theorem]{Example}
\numberwithin{equation}{section}

\definecolor{darkgreen}{RGB}{0,180,0}

\definecolor{darkred}{rgb}{0.7,0,0} 
\newcommand{\defn}[1]{{\color{darkred}\emph{#1}}} 

\newcommand{\ZZ}{\mathbb{Z}}

\newcommand{\CC}{\mathbb{C}}
\newcommand{\cc}{\mathbf{c}}

\newcommand{\zz}{\mathbf{z}}
\newcommand{\fP}{\mathfrak{P}}

\newcommand{\fsl}{\mathfrak{sl}}
\newcommand{\sym}[1]{S_{#1}}  
\newcommand{\states}{\mathfrak{S}}  
\newcommand{\model}{\mathcal{M}}  
\newcommand{\G}{\mathfrak{G}}  

\newcommand{\wt}{\operatorname{wt}}  
\newcommand{\abs}[1]{\lvert #1 \rvert}

\newcommand{\ssyt}{\operatorname{SSYT}}  
\newcommand{\svt}{\operatorname{SVT}}  
\newcommand{\rsvt}{\operatorname{RSVT}}  
\newcommand{\skyline}{\operatorname{SSLT}} 
\newcommand{\excess}{\operatorname{ex}}  
\newcommand{\stab}{\operatorname{Stab}}  
\newcommand{\Gr}{\operatorname{Gr}}  
\newcommand{\End}{\operatorname{End}}  

\usepackage{listings}
\lstdefinelanguage{Sage}[]{Python}
{morekeywords={False,sage,True},sensitive=true}
\definecolor{dblackcolor}{rgb}{0.0,0.0,0.0}
\definecolor{dbluecolor}{rgb}{0.01,0.02,0.7}
\definecolor{dgreencolor}{rgb}{0.2,0.4,0.0}
\definecolor{dgraycolor}{rgb}{0.30,0.3,0.30}

\usepackage[colorinlistoftodos]{todonotes}

\begin{document}
\title{Colored five-vertex models and Lascoux polynomials and atoms}

\author{Valentin Buciumas}
\address[V.~Buciumas]{School of Mathematics and Physics, 
The University of Queensland, 
St.\ Lucia, QLD, 4072, 
Australia}
\email{valentin.buciumas@gmail.com}
\urladdr{https://sites.google.com/site/valentinbuciumas/}

\author{Travis Scrimshaw}
\address[T.~Scrimshaw]{School of Mathematics and Physics, 
The University of Queensland, 
St.\ Lucia, QLD, 4072, 
Australia}
\email{tcscrims@gmail.com}
\urladdr{https://people.smp.uq.edu.au/TravisScrimshaw/}

\author{Katherine Weber}
\address[K.~Weber]{School of Mathematics,
University of Minnesota,
206 Church St.\ SE,
Minneapolis, MN 55455,
USA}
\email{webe0629@umn.edu}


\begin{abstract}
We construct an integrable colored five-vertex model whose partition function is a Lascoux atom based on the five-vertex model of Motegi and Sakai and the colored five-vertex model of Brubaker, the first author, Bump, and Gustafsson.
We then modify this model in two different ways to construct a Lascoux polynomial, yielding the first proven combinatorial interpretation of a Lascoux polynomial and atom.
Using this, we prove a conjectured combinatorial interpretation in terms of set-valued tableaux of a Lascoux polynomial and atom due to Pechenik and the second author.
We also prove the Monical's conjectured combinatorial interpretation of the Lascoux atom using set-valued skyline tableaux.
\end{abstract}

\maketitle

\section{Introduction}
\label{sec:intro}

Solvable lattice models are often models for simplified physical systems such as water molecules, but are known to have applications to a diverse number of mathematical fields.
By tuning the Boltzmann weights, special functions can be expressed as the partition function of the lattice model.
Then the Yang--Baxter equation can be used on the model in order to prove relations involving the functions, often simplifying intricate combinatorial or algebraic arguments.
For example, this approach was applied by Kuperberg in counting the number of alternating sign matrices using a six-vertex model~\cite{Kuperberg96}.
Similar techniques have also been used to study probabilistic models such as the (totally) asymmetric simple exclusion process~\cite{Borodin17,CP16,KMO15,KMO16,KMO16II,MS13}.

We will be focusing on the five-vertex model of Motegi and Sakai~\cite{MS13,MS14} (with a gauge transformation on the Boltzmann weights; see Remark~\ref{remark:gauge}), whose partition function is a (symmetric (or stable) $\beta$-)Grothendieck polynomial~\cite{FK94,LS82,LS83}. This was used to establish a Cauchy identity and skew decomposition for Grothendieck polynomials.
In order to define a Grothendieck polynomial $\G_{\lambda}(\zz;\beta)$, we first recall that the Schur polynomial $s_{\lambda}(\zz)$ is the (polynomial) character of the irreducible representation corresponding to the partition $\lambda$ of the special linear Lie algebra $\fsl_n$ (we refer the reader to~\cite{FH91} for more information).
Schur functions also have a geometric interpretation as the cohomology classes of Schubert varieties of the Grassmannian $\Gr(n, k)$, the set of all $k$-dimensional subspaces in $\CC^n$.
In particular, they form a basis for the cohomology ring $H^*\bigl(\Gr(n,k) \bigr)$ when we restrict to partitions that fit inside a $k \times (n-k)$ rectangle, and so $H^*\bigl(\Gr(n,k) \bigr)$ isomorphic to a projection of the ring of symmetric functions.

To improve our understanding of the Grassmannian, we can instead use a generalized cohomology theory such as connective K-theory.
By using the push-forward of the class for any Bott--Samelson resolution of a Schubert variety, we obtain a basis for the connective K-theory ring of the Grassmannian.
This basis can be given in terms of a symmetric polynomials indexed by partitions that fit inside a $k \times (n-k)$ rectangle, and these polynomials are the Grothendieck polynomials.
As such, Grothendieck polynomials are K-theory analogs of Schur functions, which are recovered when setting $\beta = 0$.
Grothendieck polynomials have been well-studied with a combinatorial interpretation using set-valued tableaux and a Littlewood--Richardson rule~\cite{Buch02}.
Recently, a crystal structure, in the sense of Kashiwara~\cite{K90,K91}, was applied to set-valued tableaux~\cite{MPS18}, recovering the expansion into Schur functions originally due to Lenart~\cite{Lenart00}.
Furthermore, a free-fermionic presentation of Grothendieck polynomials was recently given by Iwao~\cite{Iwao19}.
The equivariant K-theory of the Grassmannian was studied using integrable systems by Wheeler and Zinn-Justin~\cite{WZJ16}, yielding a construction of double Grothendieck polynomials.

There is a refinement of Schur functions that are known as key polynomials given in terms of divided difference operators~\cite{Lascoux01}.
Key polynomials are also known as Demazure characters as they can be interpreted as characters of Demazure modules, which also have crystal bases and an explicit combinatorial description~\cite{K93,L95-3} and a geometric construction~\cite{Andersen85,LMS79}.
The K-theory analog of key polynomials are the so-called Lascoux polynomials~\cite{Lascoux01}, which despite recent attention~\cite{Kirillov:notes,Monical16,MPS18,MPS18II,PS19,RY15}, do not have any known geometric or representation theoretic interpretation and have many conjectural combinatorial interpretations~\cite{Kirillov:notes,Monical16,MPS18,PS19,RY15}, some of which are known to be equivalent~\cite{Monical16,MPS18II}.

The goal of this paper is to modify the five-vertex model so that the partition function is a Lascoux polynomial.
To do this, we need an even smaller piece, the Lascoux atom~\cite{Monical16}, which is essentially the new terms that appear when taking a larger Lascoux polynomial and has a description in terms of divided difference operators.
On the solvable lattice model side, we employ the idea of Borodin and Wheeler of using a \emph{colored} lattice model~\cite{BorodinWheelerColored}, where one can then study the atoms of special functions.
Borodin and Wheeler used a colored vertex model to study nonsymmetric spin Hall--Littlewood polynomials and nonsymmetric Macdonald polynomials~\cite{BorodinWheelernsMac}.
Brubaker, Bump, the first author, and Gustafsson studied Iwahori (and parahoric) Whittaker functions on $p$-adic groups\footnote{Here Iwahori Whittaker functions should be considered as atoms for the spherical Whittaker function; the latter is given, up to a quantum factor, by a Schur polynomial by the work of Shintani, Casselman, and Shalika} using colored lattice models~\cite{BBBGIwahori}.
By modifying the colored five-vertex by these same authors~\cite{BBBG19}, our main result is the construction of an integrable colored five-vertex model based on the Motegi--Sakai five-vertex model whose partition function is a Lascoux atom.
Then by a suitable modification of our model, we obtain a Lascoux polynomial.
In fact, we provide two such modifications and show they are naturally in bijection.

As an application, we prove~\cite[Conj.~6.1]{PS19}, thus establishing the first combinatorial interpretation of Lascoux polynomials and atoms by using a notion of a Key tableau of a set-valued tableau.
We do this by refining the bijection between Gelfand--Tsetlin patterns and states of our five-vertex model to allow marking certain places as in~\cite{MPS18} in order to obtain a bijection with set-valued tableaux.
However, in order to make this weight preserving, we need to also twist by the Lusztig involution~\cite{Lenart07} (an action of the long element of the symmetric group), which requires the crystal structure on set-valued tableaux established in~\cite{MPS18}.
Another application is proving the conjectured combinatorial interpretation of~\cite[Conj.~5.2]{Monical16}.
We do this by noting our model is naturally in bijection with reverse set-valued tableaux, noting the bijection from~\cite[Thm.~2.4]{Monical16} is governed by the semistandard case of~\cite{Mason08} and adding the so-called free entries (which are just markings on certain vertices the state), and using that the semistandard case is known to give Demazure atoms~\cite{Mason08,Mason09}.

This paper is organized as follows.
In Section~\ref{sec:background}, we provide the necessary background on tableaux combinatorics, Grothendieck and Lascoux polynomials, and lattice models. 
In Section~\ref{sec:coloredmodels} we introduce a new colored lattice model and prove by using a Yang--Baxter equation, that its partition function is equal to a Lascoux atom. 
In Section~\ref{sec:proofofconjecture} we prove~\cite[Conj.~6.1]{PS19} and~\cite[Conj.~5.2]{Monical16} by using our main result.

\subsection*{Acknowledgments}

VB would like to thank Ben Brubaker, Daniel Bump, and Henrik Gustafsson for continuous useful discussion. 
TS would like to thank Kohei Motegi and Kazumitsu Sakai for invaluable discussions and explanations on their papers, in particular the Boltzmann weights of the five-vertex model that we use here.
TS would also like to thank Cara Monical, Tomoo Matsumura, Oliver Pechenik, and Shogo Sugimoto for useful discussions.
KW would like to thank Ben Brubaker for useful discussion and inspiration.
The authors thank Oliver Pechenik for comments on an earlier draft of this manuscript.
The authors thank the anonymous referees for their helpful comments.
This paper benefited from computations using \textsc{SageMath}~\cite{sage}.

VB was partially supported by the Australian Research Council DP180103150.
TS was partially supported by the Australian Research Council DP170102648.

\section{Background}
\label{sec:background}

Fix a positive integer $n$.
Let $\zz = (z_1, z_2, \dotsc, z_n)$ be a finite number of indeterminates.
For any sequence $(\alpha_1, \dotsc, \alpha_n) \in \ZZ^n$ of length $n$, denote $\zz^{\alpha} := z_1^{\alpha_1} z_2^{\alpha_2} \cdots z_n^{\alpha_n}$.
Let $\sym{n}$ denote the symmetric group on $n$ elements with simple transpositions $(s_1, s_2, \dotsc, s_{n-1})$.
For some $w \in \sym{n}$, let $\ell(w)$ denote the length of $w$: the minimal number of simple transpositions whose product equals $w$.
We denote by $w_0$ the longest element in  $S_n$. 
Let $\leq$ denote the (strong) Bruhat order on $S_n$.
For more information on the symmetric group, we refer the reader to~\cite{Sagan01}.
For any $w \in \sym{n}$, define $w\zz = (z_{w(1)}, z_{w(2)}, \dotsc, z_{w(n)})$.
Let $\lambda = (\lambda_1, \lambda_2, \dotsc, \lambda_n)$ be a \defn{partition}, a sequence of weakly decreasing nonnegative integers (of length $n$).
Let $\ell(\lambda) = \max \{k \mid \lambda_k > 0 \}$ denote the \defn{length} of $\lambda$.
The Young diagram (in English convention) of $\lambda$ is a drawing consisting of stacks of boxes with row $i$ having $\lambda_i$ boxes pushed into the upper-left corner.

\subsection{Tableaux combinatorics}

A \defn{semistandard (Young) tableau of shape $\lambda$} is a filling of the boxes of the Young diagram of $\lambda$ with positive integers such that the values are weakly increasing across rows and strictly increasing down columns.
Let $\ssyt^n(\lambda)$ denote the set of semistandard Young tableaux of shape $\lambda$ and maximum entry $n$.
A \defn{(semistandard) set-valued tableau of shape $\lambda$} is similar except we fill the boxes with finite non-empty sets of positive integers that satisfy
\[
\ytableausetup{boxsize=1.5em}
\ytableaushort{XY,Z}\ \ \text{ implies } \max X \leq \min Y \text{ and } \max X < \min Z.
\]
Let $\svt^n(\lambda)$ denote the set of all set-valued tableaux of shape $\lambda$ such that the maximum integer appearing is $n$.
We equate a set-valued tableau where every entry has size $1$ with the semistandard Young tableau obtained by forgetting the entry is a set.

Define the \defn{weight} of a set-valued tableau $T \in \svt^n(\lambda)$ to be
\[
\wt(T) := \bigl( \#\{ X \in T \mid i \in X \} \bigr)_{i=1}^n;
\]
in particular, the exponent of $z_i$ in $\zz^{\wt(T)}$ counts the number of times that $i$ occurs in (an entry of) $T$.
We also require the \defn{excess} statistic:
\[
\excess(T) := \sum_{X \in T} \bigl( |X| - 1 \bigr),
\]
which counts how far the set-valued tableau $T$ is from being a semistandard Young tableau (\textit{i.e.}, $T$ is a semistandard Young tableau if and only if $\excess(T) = 0$).

A semistandard Young tableau is called a \defn{key tableau} if the entries of column $i + 1$ are a subset of the entries of column $i$ for all $1 \leq i < \lambda_1$.
We define a left $\sym{n}$-action on key tableau $K$ with maximum entry $n$ by applying $w \in \sym{n}$ to each entry of $K$ and sorting columns to be strictly increasing.
Let $K_{w\lambda}$ denote the key tableau by applying $w$ to the key tableau of shape $\lambda$ with every entry of row $i$ filled by $i$.
Note that this agrees with $\sym{n}$ considered as the Weyl group acting on the crystal of $\ssyt^n(\lambda)$ (we refer the reader to~\cite{BS17} for more details).

Let $T$ be a set-valued tableau.
For a semistandard Young tableau $S$, let $k(S)$ denote the (right) key tableau associated to $S$ (see, \textit{e.g.},~\cite{Willis13,BBBG19} for algorithms to compute this).
Let $\min(T)$ denote the semistandard Young tableau formed by taking the minimum of each entry in $T$.
Let $T^*$ denote the \defn{Lusztig involution} on $T$ using the crystal structure from~\cite{MPS18}.
We do not require the exact definition of the Lusztig involution, only that $\wt(T^*) = w_0 \wt(T)$.
From~\cite[Sec.~6]{PS19}, we define the \defn{(right) Key tableau} of $T$ to be
\[
K(T) := k\bigl( \min(T^*)^* \bigr).
\]

\begin{example}
Let $\lambda = (4,2,1)$ and $n = 3$. Consider the tableau
\[
T = \ytableaushort{111{2,\!3},22,3}\,,
\]
then computing $K(T)$, we have
\[
T^* = \ytableaushort{11{2,\!3}3,22,3}\,,
\qquad
\min(T^*) = \ytableaushort{1123,22,3}\,,
\qquad
\min(T^*)^* = \ytableaushort{1122,23,3}\,,
\]
for which we then compute the key (for semistandard Young tableaux) to obtain the Key
\[
K(T) = k\bigl(\min(T^*)^* \bigr) = \ytableaushort{1222,23,3}\,.
\]
\end{example}

Now we recall the definition of a \defn{marked Gelfand--Tsetlin (GT) pattern} and the bijection with set-valued tableaux from~\cite[Sec.~4]{MPS18}.
Indeed, a marked GT pattern is a sequence of partitions $\Lambda = (\lambda^{(j)})_{j=0}^n$, called a \defn{Gelfand--Tsetlin (GT) pattern}, such that $\lambda^{(0)} = \emptyset$ and the skew shape $\lambda^{(j)} / \lambda^{(j-1)}$ does not contain a vertical domino (\textit{i.e.}, is a horizontal strip),\footnote{This is equivalent to the usual interlacing condition on GT patterns.} with a set $M$ of entries that are ``marked,'' where the entry $(i, j)$, for $2 \leq j \leq n$ and $1 \leq i < \ell(\lambda^{(j)})$, is allowed to be marked if and only if $\lambda^{(j)}_{i+1} < \lambda^{(j-1)}_i$.
In particular, an entry $(i, j)$ cannot be marked if the entry to the right equals the entry to the southeast.
We depict a marked GT pattern as a triangular array with the top-row corresponding to $\lambda^{(n)}$ and the bottom row $\lambda^{(1)}$ and a marked entry $(i,j)$ as a box around the entry $\lambda^{(j)}_i$.

Next, we recall the bijection $\phi$ between marked GT patterns and set-valued tableaux, which is defined recursively as follows.
Consider a marked GT pattern $(\Lambda, M)$.
Start with $T_0 = \emptyset$.
Suppose we are at step $j$, where the set-valued tableau is $T_{j-1}$ that has entries in $1, \dotsc, j-1$.
For each marked entry $(i,j)$, we add $j$ to the rightmost entry of $i$-th row of $T_{j-1}$, and denote this $T'_j$.
Then we consider the horizontal strip $\lambda^{(j)} / \lambda^{(j-1)}$ with all entries being $\{j\}$, which we add to $T'_j$ to obtain a set-valued tableau $T_j$ of shape $\lambda^{(j)}$.
We repeat this for every row of $\Lambda$ and the result is $\phi(\Lambda, M)$.
We define the weight of a marked GT pattern $\wt(\Lambda, M) = \wt\bigl( \phi(\Lambda, M) \bigr)$.

\begin{example}
We depict the marked entries $M$ in a marked GT pattern $(\Lambda, M)$ by boxing the marked entries and we underline entries not allowed to be marked.
Consider $\lambda = (7, 4, 4)$ and $n = 4$. Then an example of a marked GT pattern with top row $\lambda$ and the corresponding set-valued tableau under $\phi$ is
\[
(\Lambda, M) =
\begin{array}{ccccccc}
\boxed{7} && \underline{5} && 4 && \underline{0} \\[4pt]
& 6 && \boxed{4} && \underline{2} \\[4pt]
&& \boxed{6} && \underline{4} \\[4pt]
&&& \underline{6}
\end{array},
\quad
\ytableausetup{boxsize=2.2em}
\phi(\Lambda, M) = 
\ytableaushort{11111{1,\!2,\!4}4,222{2,\!3},3{3,\!4}44}\,.
\]
\end{example}

We also require one additional combinatorial object from~\cite{Monical16}, where we use the description given in~\cite{MPS18II} but converted to use English convention.
For a permutation $w \in \sym{n}$, define the \defn{(semistandard) skyline diagram} $w\lambda$ to be the Young diagram of $\lambda$ but with the rows permuted by $w$.
In particular, we have $w \lambda = (\lambda_{w(1)}, \dotsc, \lambda_{w(n)})$.
A \defn{set-valued skyline tableau of shape $w\lambda$} is a filling of a skyline diagram $w\lambda$ with finite nonempty sets of positive integers that satisfy the following conditions. Call the largest entry in a box the \defn{anchor} and the other entries \defn{free}.
\begin{enumerate}
\item Entries do not repeat in a column.
\item \label{weak_increase} Rows weakly decrease in the sense of if $B$ is to the left of $A$, then $\min B \geq \max A$.
\item For every triple of boxes of the form
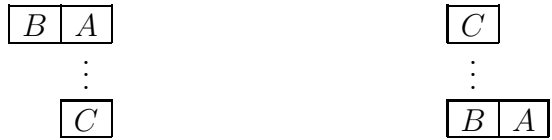
\[
\begin{array}{c@{\hspace{40pt}}c}
\begin{array}{|c|c|}
\cline{1-2}
\raisebox{-2pt}{$B$} & \raisebox{-2pt}{$A$}
\\ \cline{1-2} \multicolumn{1}{c}{} & \multicolumn{1}{c}{\vdots}
\\ \cline{2-2}
\multicolumn{1}{c|}{} & \raisebox{-2pt}{$C$}
\\ \cline{2-2}
\end{array}
&
\begin{array}{|c|c|}
\cline{1-1}
\raisebox{-2pt}{$C$}
\\ \cline{1-1} \multicolumn{1}{c}{\vdots}
\\ \cline{1-2}
\raisebox{-2pt}{$B$} & \raisebox{-2pt}{$A$}
\\ \cline{1-2}
\end{array}
\\[25pt]
\text{upper row weakly longer}
&
\text{lower row strictly longer}
\end{array}
\]
the anchors $a,b,c$ of $A,B,C$, respectively, must satisfy either $c < a$ or $b < c$.\footnote{Such triples were originally called \defn{inversion triples} and required to satisfy $c < a \leq b$ or $a \leq b < c$, but in our case $a \leq b$ is immediate by~(\ref{weak_increase}).}
\item \label{floating_free} Every free entry is in the cell of the least anchor in its column such that~(\ref{weak_increase}) is not violated.
\item \label{anchors} Anchors in the first column equal their row index.
\end{enumerate}
Let $\skyline(w\lambda)$ denote the set of set-valued skyline tableaux of shape $w\lambda$.
We define the weight and excess for a set-valued skyline tableau the same way as for a set-valued tableau.

\begin{example}
The set-valued skyline tableaux in the set $\skyline(s_1 s_2 \lambda)$ for $\lambda = (2, 4, 1)$ are
\[
\ytableausetup{boxsize=1.5em}
\begin{array}{c@{\quad}c@{\quad}c@{\quad}c@{\quad}c}
\ytableaushort{1,2{2,\!1}11,33}\,,
&
\ytableaushort{1,22{2,\!1}1,3{3,\!1}}\,,
&
\ytableaushort{1,22{2,\!1}1,33}\,,
&
\ytableaushort{1,222{2,\!1},3{3,\!1}}\,,
&
\ytableaushort{1,222{2,\!1},33}\,,
\\[30pt]
\ytableaushort{1,2211,33}\,,
&
\ytableaushort{1,2221,3{3,\!1}}\,,
&
\ytableaushort{1,2221,33}\,,
&
\ytableaushort{1,2222,3{3,\!1}}\,,
&
\ytableaushort{1,2222,33}\,.
\end{array}
\]
\end{example}

\subsection{Symmetric Grothendieck polynomials and Lascoux polynomials}

For the remainder of this section, we will consider $\lambda$ to always be a partition of length $n$, but possibly with some entries being $0$.

The classical definition of a \defn{Schur polynomial} is given as a ratio of determinants:
\[
s_{\lambda}(\zz) = \dfrac{\det\bigl( z_j^{\lambda_i+n-i} \bigr)_{i,j=1}^n}{\prod_{1\leq i < j \leq n} z_i - z_j},
\]
where the denominator is the Vandermonde determinant (equivalently, the determinant in the numerator when $\lambda = \emptyset$).
Ikeda and Naruse~\cite{IN13} gave a similar definition for the \defn{(symmetric)\footnote{These are also known as \defn{stable Grothendieck polynomials} as they are the stable limits $n \to \infty$ of Grothendieck polynomials and then restricted to a finite number of variables when $\beta = -1$~\cite{LS82,LS83}.} Grothendieck polynomial}:
\begin{equation}
\label{eq:determinant_formula}
\G_{\lambda}(\zz; \beta) = \dfrac{\det\bigl( z_j^{\lambda_i+n-i}(1-\beta z_i)^{j-1} \bigr)_{i,j=1}^n}{\prod_{1\leq i < j \leq n} z_i - z_j}.
\end{equation}
Note that we have $\G_{\lambda}(\zz; 0) = s_{\lambda}(\zz)$.
There is a combinatorial interpretation of Grothendieck polynomials due to Buch~\cite[Thm.~3.1]{Buch02} as the generating function of semistandard set-valued tableaux:
\[
\G_{\lambda}(\zz; \beta) = \sum_{T \in \svt^n(\lambda)} \beta^{\excess(T)} \zz^{\wt(T)}.
\]
Using $\phi$, we can give another combinatorial interpretation of a Grothendieck polynomial as the generating function over marked GT patterns~\cite[Prop.~4.5]{MPS18}.

Finally, we will consider another algebraic definition of the Grothendieck polynomials.
The \defn{Demazure--Lascoux operator} $\varpi_i$ defines an action of the $0$-Hecke algebra on $\ZZ[\beta][\zz]$ given by 
\[
\varpi_i f(\zz; \beta) = \frac{(z_i  + \beta z_i z_{i+1}) f(\zz; \beta) -  (z_{i+1}  + \beta z_i z_{i+1}) f(s_i \zz; \beta)}{z_i - z_{i+1}}.
\] 
In particular, the Demazure--Lascoux operators satisfy the relations:
\begin{subequations}
\label{eq:0Hecke}
\begin{align}
\label{eq:DL_braid} \varpi_i \varpi_j &= \varpi_j \varpi_i \hspace{1.5cm} \text{for $|i-j| > 1$}, \\
\label{eq:DL_comm} \varpi_i \varpi_{i+1} \varpi_i &= \varpi_{i+1} \varpi_i \varpi_{i+1}, \\
\label{eq:DL_0Hecke} \varpi_i^2 & = \varpi_i.
\end{align}
\end{subequations}
Hence, for any permutation $w \in \sym{n}$, one may define $\varpi_w := \varpi_{i_1} \varpi_{i_2} \dotsm \varpi_{i_{\ell}}$ for any choice of reduced expression $w = s_{i_1} s_{i_2} \dotsm s_{i_{\ell}}$.
We define the \defn{Lascoux polynomials}~\cite{Lascoux01} by
\[
L_{w\lambda}(\zz; \beta) := \varpi_w \zz^{\lambda},
\]
and we have $\G_{\lambda}(\zz; \beta) = L_{w_0\lambda}(\zz; \beta)$~\cite{IN13,IS14,LS82}.
It is an open problem to find a geometric or representation-theoretic interpretation for general Lascoux polynomials.

Next, following~\cite{Monical16}, we define the \defn{Demazure--Lascoux atom operator} $\overline{\varpi}_i := \varpi_i - 1$, which satisfies the relations~\eqref{eq:0Hecke} except $\overline{\varpi}_i^2 = -\overline{\varpi}_i$ instead of~\eqref{eq:DL_0Hecke}.
These operators are used to define the \defn{Lascoux atoms}:
\[
\overline{L}_{w\lambda}(\zz; \beta) := \overline{\varpi}_w \zz^{\lambda},
\]
which then satisfy~\cite[Thm.~5.1]{Monical16}
\begin{equation}
\label{eq:Lascoux_into_atoms}
L_{w\lambda}(\zz; \beta) = \sum_{u \leq w} \overline{L}_{u\lambda}(\zz; \beta).
\end{equation}
When $\beta = 0$, Lascoux and Sch\"utzenberger~\cite{LS90} showed that
\[
\overline{L}_{w\lambda}(\zz; 0) = \sum_{\substack{T \in \svt^n(\lambda) \\ K(T) = K_{w\lambda}}} \zz^{\wt(T)},
\qquad\qquad
L_{w\lambda}(\zz; 0) = \sum_{\substack{T \in \svt^n(\lambda) \\ K(T) \leq K_{w\lambda}}} \zz^{\wt(T)},
\]
where the comparison $K(T) \leq K_{w\lambda}$ is done entrywise.
It is straightforward to see that $K(T) \leq K_{w\lambda}$ is equivalent to saying there exists a $u \leq w$ such that $K(T) = K_{u\lambda}$.

\begin{conjecture}[{\cite[Conj.~6.1]{PS19}}]
\label{conj:svt_Lascoux}
We have
\[
\overline{L}_{w\lambda}(\zz; \beta) = \sum_{\substack{T \in \svt^n(\lambda) \\ K(T) = K_{w\lambda}}} \beta^{\excess(T)} \zz^{\wt(T)},
\qquad\qquad
L_{w\lambda}(\zz; \beta) = \sum_{\substack{T \in \svt^n(\lambda) \\ K(T) \leq K_{w\lambda}}} \beta^{\excess(T)} \zz^{\wt(T)}.
\]
\end{conjecture}

The following conjecture is also the K-theoretic analog of Demazure characters being described by skyline tableaux~\cite{Mason08,Mason09}, which come from nonsymmetric Macdonald polynomials at $t = q = 0$.

\begin{conjecture}[{\cite[Conj.~5.2]{Monical16}}]
\label{conj:skyline_tableaux}
We have
\[
\overline{L}_{w\lambda} = \sum_{S \in \skyline(w\lambda)} \beta^{\excess(S)} \zz^{\wt(S)}.
\]
\end{conjecture}

We note that if $\varpi_i f = f$, then $\overline{\varpi}_i f = \varpi_i f - f = f - f = 0$.
Hence if there exists an $i$ such that $\ell(ws_i) = \ell(w) - 1$ and $\varpi_i \zz^{\lambda} = \zz^{\lambda}$, then $\overline{\varpi}_w \zz^{\lambda} = 0$.
This condition on $w$ is equivalent to there being a $u$ such that $\ell(u) < \ell(w)$ and $u \lambda = w \lambda$.
The elements $w$ of minimal length such that $w\lambda \neq \lambda$ are the \defn{minimal length coset representatives} of $\sym{n} / \stab(\lambda)$, where $\stab(\lambda) = \{ w \in \sym{n} \mid w\lambda = \lambda \}$ is the stabilizer of $\lambda$ (for more information, see standard texts such as~\cite{BB05,Humphreys90}).

\subsection{Uncolored models}
\label{sec:uncolored_model}

Now we give another interpretation of the Grothendieck polynomial $\G_{\lambda}$ using integrable systems due to Motegi and Sakai~\cite[Lemma~5.2]{MS13}.
We first equate $\lambda$ with a $\{0,1\}$-sequence of length $m$ by considering the Young diagram inside of an $n \times (m - n)$ rectangle and starting at the bottom left, each up step we write a $1$ and each right step we write a $0$.
Note that the positions of the $1$'s are at $\lambda_i + i$ for all $1 \leq i \leq n$.
For example, with $\lambda = 522100$ (so $n = 6$) with $m = 14$, the corresponding $\{0,1\}$-sequence is $11010110001000$.

\begin{figure}
\[
\begin{array}{ccccc}
\toprule
  \tt{a}_1&\tt{a}_2&\tt{b}_1&\tt{b}_2&\tt{c}_1\\
\midrule
\begin{tikzpicture}
\coordinate (a) at (-.75, 0);
\coordinate (b) at (0, .75);
\coordinate (c) at (.75, 0);
\coordinate (d) at (0, -.75);
\coordinate (aa) at (-.75,.5);
\coordinate (cc) at (.75,.5);
\draw (a)--(0,0);
\draw (b)--(0,0);
\draw (c)--(0,0);
\draw (d)--(0,0);
\draw[fill=white] (a) circle (.25);
\draw[fill=white] (b) circle (.25);
\draw[fill=white] (c) circle (.25);
\draw[fill=white] (d) circle (.25);
\node at (0,1) { };
\node at (a) {$0$};
\node at (b) {$0$};
\node at (c) {$0$};
\node at (d) {$0$};
\path[fill=white] (0,0) circle (.2);
\node at (0,0) {$z$};
\end{tikzpicture} &
\begin{tikzpicture}
\coordinate (a) at (-.75, 0);
\coordinate (b) at (0, .75);
\coordinate (c) at (.75, 0);
\coordinate (d) at (0, -.75);
\coordinate (aa) at (-.75,.5);
\coordinate (cc) at (.75,.5);
\draw (a)--(0,0);
\draw (b)--(0,0);
\draw[line width=0.5mm] (c)--(0,0);
\draw[line width=0.5mm] (d)--(0,0);
\draw[fill=white] (a) circle (.25);
\draw[fill=white] (b) circle (.25);
\draw[fill=white] (c) circle (.25);
\draw[fill=white] (d) circle (.25);
\node at (0,1) { };
\node at (a) {$0$};
\node at (b) {$0$};
\node at (c) {$1$};
\node at (d) {$1$};
\path[fill=white] (0,0) circle (.2);
\node at (0,0) {$z$};
\end{tikzpicture}&
\begin{tikzpicture}
\coordinate (a) at (-.75, 0);
\coordinate (b) at (0, .75);
\coordinate (c) at (.75, 0);
\coordinate (d) at (0, -.75);
\coordinate (aa) at (-.75,.5);
\coordinate (cc) at (.75,.5);
\draw[line width=0.5mm] (a)--(0,0);
\draw[line width=0.5mm] (b)--(0,0);
\draw[line width=0.5mm] (c)--(0,0);
\draw[line width=0.5mm] (d)--(0,0);
\draw[fill=white] (a) circle (.25);
\draw[fill=white] (b) circle (.25);
\draw[fill=white] (c) circle (.25);
\draw[fill=white] (d) circle (.25);
\node at (0,1) { };
\node at (a) {$1$};
\node at (b) {$1$};
\node at (c) {$1$};
\node at (d) {$1$};
\path[fill=white] (0,0) circle (.2);
\node at (0,0) {$z$};
\end{tikzpicture}&
\begin{tikzpicture}
\coordinate (a) at (-.75, 0);
\coordinate (b) at (0, .75);
\coordinate (c) at (.75, 0);
\coordinate (d) at (0, -.75);
\coordinate (aa) at (-.75,.5);
\coordinate (cc) at (.75,.5);
\draw[line width=0.5mm] (a)--(0,0);
\draw (b)--(0,0);
\draw[line width=0.5mm] (c)--(0,0);
\draw (d)--(0,0);
\draw[fill=white] (a) circle (.25);
\draw[fill=white] (b) circle (.25);
\draw[fill=white] (c) circle (.25);
\draw[fill=white] (d) circle (.25);
\node at (0,1) { };
\node at (a) {$1$};
\node at (b) {$0$};
\node at (c) {$1$};
\node at (d) {$0$};
\path[fill=white] (0,0) circle (.2);
\node at (0,0) {$z$};
\end{tikzpicture}&
\begin{tikzpicture}
\coordinate (a) at (-.75, 0);
\coordinate (b) at (0, .75);
\coordinate (c) at (.75, 0);
\coordinate (d) at (0, -.75);
\coordinate (aa) at (-.75,.5);
\coordinate (cc) at (.75,.5);
\draw[line width=0.5mm] (a)--(0,0);
\draw[line width=0.5mm] (b)--(0,0);
\draw (c)--(0,0);
\draw (d)--(0,0);
\draw[fill=white] (a) circle (.25);
\draw[fill=white] (b) circle (.25);
\draw[fill=white] (c) circle (.25);
\draw[fill=white] (d) circle (.25);
\node at (0,1) { };
\node at (a) {$1$};
\node at (b) {$1$};
\node at (c) {$0$};
\node at (d) {$0$};
\path[fill=white] (0,0) circle (.2);
\node at (0,0) {$z$};
\end{tikzpicture} \\
\midrule
  1&1+\beta z&1&z&1\\
\bottomrule
\end{array}\]
\caption{Boltzmann weights of the uncolored model.}
\label{fig:uncolored_wts}
\end{figure}

Next, we consider a rectangular grid with $n$ horizontal lines and $m$ vertical lines and to each (half) edge (we consider a crossing of the lines to be vertices), we assign either a $0$ or a $1$.
For the \defn{(lattice) model} we will be considering, we fix the left (half) edges to all have label by $1$, the right and bottom (half) edges all being labeled by $0$, and the top (half) edges given by the $\{0,1\}$-sequence corresponding to $\lambda$.
A \defn{state} is an assignment of $\{0,1\}$ labels to all of the remaining edges of the model.
We call a state \defn{admissible} if all of the local configurations around each vertex are one of the configurations given by Figure~\ref{fig:uncolored_wts}, each of which has a \defn{Boltzmann weight}.
Let $\states_{\lambda}$ denote the set of all possible admissible states of the model.
The \defn{(Boltzmann) weight} $\wt(S)$ of an admissible state $S \in \states_{\lambda}$ is the product of all of the Boltzmann weights of all vertices with $z = z_i$ in the $i$-th row numbered starting from top.
We consider the Boltzmann weight of an inadmissible state to be $0$.
The \defn{partition function} of a model $\model$ (\textit{i.e.}, a set of (admissible) states)
\[
Z(\model; \zz; \beta) := \sum_{S \in \model} \wt(S)
\]
is the sum of the Boltzmann weights of all possible (admissible) states of $\model$.

\begin{theorem}[{\cite[Lemma~5.2]{MS13}}]
We have
\[
\G_{\lambda}(\zz; \beta) = Z(\states_{\lambda}; \zz; \beta).
\]
\end{theorem}

\begin{remark}
\label{remark:gauge}
While the result in~\cite{MS13} was given in terms of wavefunctions, the pictorial description makes it clear that this is the same as computing the partition function of $\states_{\lambda}$.
Indeed, an admissible configuration around a vertex can be considered as the $L$-matrix $L \in \End( W_a \otimes V_j )$, where $W_a = \CC^2$ is the $a$-th auxiliary space and $V_j$ is the $j$-th quantum space.
Furthermore, we have also taken a gauge transformation by $z_i = -\beta^{-1} - u_i^{-2}$.
\end{remark}

One important aspect of the model $\states_{\lambda}$ is every admissible state corresponds to a GT pattern $(\lambda^{(i)})_{i=0}^n$ by letting $\lambda^{(i)}$ be the $(n-i)$-th row of vertical edges (with the top (half) edges being the $0$-th row of the model) to be the $\{0,1\}$-sequence of a partition~\cite[Sec.~3]{MS14}.
We denote this bijection $\fP$ from $\states_{\lambda}$ to GT patterns with top row $\lambda$.
However, the bijection $\fP$ is not weight preserving, but instead we have to do also apply the map $z_i \mapsto z_{n+1-i}$.
Therefore, it is straightforward to see that for any $S \in \states_{\lambda}$, we have
\begin{equation}
\label{eq:wt_by_GT_patterns}
\wt(S) = \sum_{(\fP(S), M)} \beta^{\abs{M}} w_0 \wt(\fP(S), M),
\end{equation}
where we sum over all possible markings of $\fP(S)$.
Note that we have freedom for $\tt{a}_2$ vertices in a state to be or not be marked.
In particular, a state is unmarked if and only if it corresponds to a semistandard Young tableau.
Similarly for GT patterns.

\begin{example}
The following is an admissible state $S \in \states_{221}$ with $n = 3$ and $m = 5$:
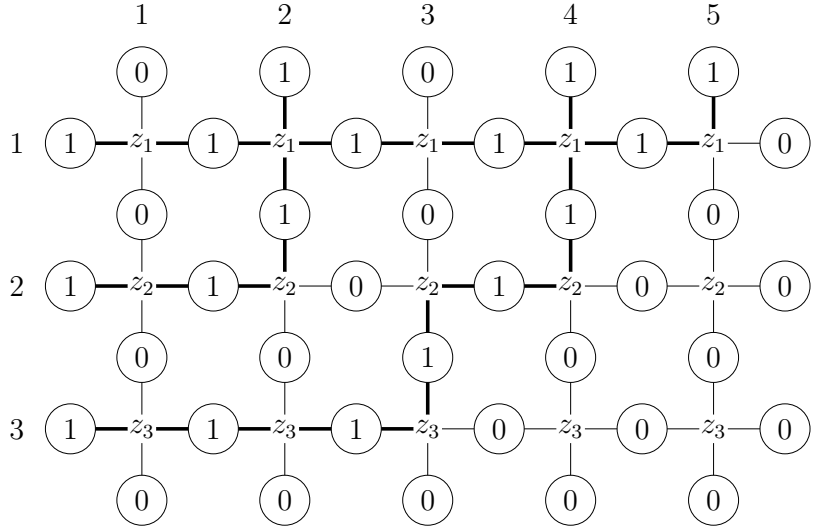
\[
  \scalebox{.95}{\begin{tikzpicture}
    \draw[line width=0.5mm] (9,6) -- (9,5) -- (7,5) -- (7,3) -- (5,3) -- (5,1) -- (0,1);
    \draw[line width=0.5mm] (7,6) -- (7,5) -- (0,5);
    \draw[line width=0.5mm] (3,6) -- (3,3) -- (0,3);
    \draw (0,1)--(10,1);
    \draw (0,3)--(10,3);
    \draw (0,5)--(10,5);
    \foreach \x in {0,2,4,6,8,10}
    {
      \draw[fill=white] (\x,5) circle (.35);
      \draw[fill=white] (\x,3) circle (.35);
      \draw[fill=white] (\x,1) circle (.35);
    }
    \foreach \x in {1,3,5,7,9}
    {
      \draw (\x,0)--(\x,6);
      \draw[fill=white] (\x,0) circle (.35);
      \draw[fill=white] (\x,2) circle (.35);
      \draw[fill=white] (\x,4) circle (.35);
      \draw[fill=white] (\x,6) circle (.35);
      \path[fill=white] (\x,1) circle (.2);
      \node at (\x,1) {$z_3$};
      \path[fill=white] (\x,3) circle (.2);
      \node at (\x,3) {$z_2$};
      \path[fill=white] (\x,5) circle (.2);
      \node at (\x,5) {$z_1$};
    }

    \node at (1,6) {$0$};
    \node at (3,6) {$1$};
    \node at (5,6) {$0$};
    \node at (7,6) {$1$};
    \node at (9,6) {$1$};
    \node at (1,4) {$0$};
    \node at (3,4) {$1$};
    \node at (5,4) {$0$};
    \node at (7,4) {$1$};
    \node at (9,4) {$0$};
    \node at (1,2) {$0$};
    \node at (3,2) {$0$};
    \node at (5,2) {$1$};
    \node at (7,2) {$0$};
    \node at (9,2) {$0$};
    \node at (1,0) {$0$};
    \node at (3,0) {$0$};
    \node at (5,0) {$0$};
    \node at (7,0) {$0$};
    \node at (9,0) {$0$};
    \node at (0,5) {$1$};
    \node at (2,5) {$1$};
    \node at (4,5) {$1$};
    \node at (6,5) {$1$};
    \node at (8,5) {$1$};
    \node at (10,5) {$0$};
    \node at (0,3) {$1$};
    \node at (2,3) {$1$};
    \node at (4,3) {$0$};
    \node at (6,3) {$1$};
    \node at (8,3) {$0$};
    \node at (10,3) {$0$};
    \node at (0,1) {$1$};
    \node at (2,1) {$1$};
    \node at (4,1) {$1$};
    \node at (6,1) {$0$};
    \node at (8,1) {$0$};
    \node at (10,1) {$0$};
    \node at (1.00,6.8) {$ 1$};
    \node at (3.00,6.8) {$ 2$};
    \node at (5.00,6.8) {$ 3$};
    \node at (7.00,6.8) {$ 4$};
    \node at (9.00,6.8) {$ 5$};
    \node at (-.75,1) {$ 3$};
    \node at (-.75,3) {$ 2$};
    \node at (-.75,5) {$ 1$};
  \end{tikzpicture}}
\]
The Boltzmann weight of this state is
\[
\wt(S) = z_1^2 z_2 (1 + \beta z_2) z_3^2 = z_1^2 z_2 z_3^2 + \beta z_1^2 z_2^2 z_3^2,
\]
and the corresponding GT pattern is $\fP(S) = (\emptyset, 2, 21, 221)$.
There are two possible markings for $\fP(S)$, which correspond to the set-valued tableaux
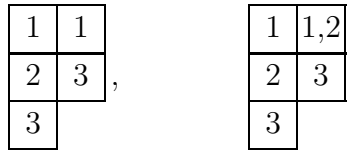
\[
\ytableaushort{11,23,3}\,, \qquad\qquad \ytableaushort{1{1,\!2},23,3}\,.
\]
The left tableaux is also a semistandard Young tableau.
\end{example}

\begin{figure}
\[
\begin{array}{c@{\qquad}c@{\qquad}c@{\qquad}c@{\qquad}c}
\toprule
\begin{tikzpicture}[scale=0.7]
\draw (0,0) to [out = 0, in = 180] (2,2);
\draw (0,2) to [out = 0, in = 180] (2,0);
\draw[fill=white] (0,0) circle (.35);
\draw[fill=white] (0,2) circle (.35);
\draw[fill=white] (2,0) circle (.35);
\draw[fill=white] (2,2) circle (.35);
\node at (0,0) {$1$};
\node at (0,2) {$1$};
\node at (2,2) {$1$};
\node at (2,0) {$1$};
\node at (2,0) {$1$};
\path[fill=white] (1,1) circle (.3);
\node at (1,1) {$z_i,z_j$};
\end{tikzpicture}&
\begin{tikzpicture}[scale=0.7]
\draw (0,0) to [out = 0, in = 180] (2,2);
\draw (0,2) to [out = 0, in = 180] (2,0);
\draw[fill=white] (0,0) circle (.35);
\draw[fill=white] (0,2) circle (.35);
\draw[fill=white] (2,0) circle (.35);
\draw[fill=white] (2,2) circle (.35);
\node at (0,0) {$0$};
\node at (0,2) {$0$};
\node at (2,2) {$0$};
\node at (2,0) {$0$};
\path[fill=white] (1,1) circle (.3);
\node at (1,1) {$z_i,z_j$};
\end{tikzpicture}&
\begin{tikzpicture}[scale=0.7]
\draw (0,0) to [out = 0, in = 180] (2,2);
\draw (0,2) to [out = 0, in = 180] (2,0);
\draw[fill=white] (0,0) circle (.35);
\draw[fill=white] (0,2) circle (.35);
\draw[fill=white] (2,0) circle (.35);
\draw[fill=white] (2,2) circle (.35);
\node at (0,0) {$0$};
\node at (0,2) {$1$};
\node at (2,2) {$0$};
\node at (2,0) {$1$};
\path[fill=white] (1,1) circle (.3);
\node at (1,1) {$z_i,z_j$};
\end{tikzpicture}&
\begin{tikzpicture}[scale=0.7]
\draw (0,0) to [out = 0, in = 180] (2,2);
\draw (0,2) to [out = 0, in = 180] (2,0);
\draw[fill=white] (0,0) circle (.35);
\draw[fill=white] (0,2) circle (.35);
\draw[fill=white] (2,0) circle (.35);
\draw[fill=white] (2,2) circle (.35);
\node at (0,0) {$0$};
\node at (0,2) {$1$};
\node at (2,2) {$1$};
\node at (2,0) {$0$};
\path[fill=white] (1,1) circle (.3);
\node at (1,1) {$z_i,z_j$};
\end{tikzpicture}&
\begin{tikzpicture}[scale=0.7]
\draw (0,0) to [out = 0, in = 180] (2,2);
\draw (0,2) to [out = 0, in = 180] (2,0);
\draw[fill=white] (0,0) circle (.35);
\draw[fill=white] (0,2) circle (.35);
\draw[fill=white] (2,0) circle (.35);
\draw[fill=white] (2,2) circle (.35);
\node at (0,0) {$1$};
\node at (0,2) {$0$};
\node at (2,2) {$0$};
\node at (2,0) {$1$};
\path[fill=white] (1,1) circle (.3);
\node at (1,1) {$z_i,z_j$};
\end{tikzpicture}\\
\midrule
(1 + \beta z_i) z_j & (1 + \beta z_i) z_j & (z_j - z_i) z_j  & (1 + \beta z_i) z_j & (1 + \beta z_j) z_j \\
\bottomrule
\end{array}
\]
\caption{The $R$-matrix for the uncolored model.}
\label{fig:uncolored_R_matrix}
\end{figure}
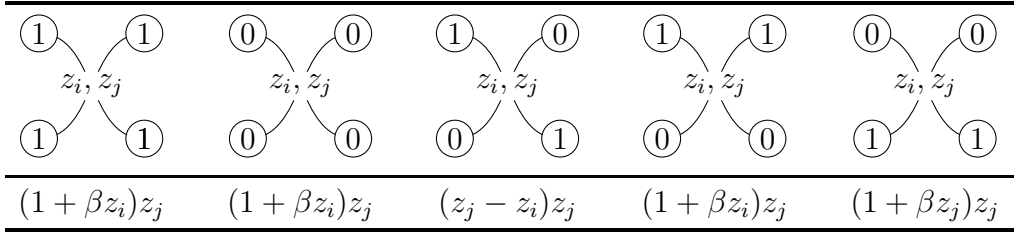

We note that this model is integrable, in the sense that the $R$-matrix given by Figure~\ref{fig:uncolored_R_matrix} satisfies the $RLL$-relation (a version of the Yang--Baxter equation).

\begin{proposition}[{\cite{MS13}}]
\label{prop:uncolored_YBE}
The partition function of the following two models are equal for any boundary conditions $a,b,c,d,e,f \in \{ 0, 1 \}$:
\begin{equation}
\label{eq:RLL_relation}
\begin{tikzpicture}[baseline=(current bounding box.center)]
  \draw (0,1) to [out = 0, in = 180] (2,3) to (4,3);
  \draw (0,3) to [out = 0, in = 180] (2,1) to (4,1);
  \draw (3,0) to (3,4);
  \draw[fill=white] (0,1) circle (.3);
  \draw[fill=white] (0,3) circle (.3);
  \draw[fill=white] (3,4) circle (.3);
  \draw[fill=white] (4,3) circle (.3);
  \draw[fill=white] (4,1) circle (.3);
  \draw[fill=white] (3,0) circle (.3);
  \node at (0,1) {$a$};
  \node at (0,3) {$b$};
  \node at (3,4) {$c$};
  \node at (4,3) {$d$};
  \node at (4,1) {$e$};
  \node at (3,0) {$f$};
\path[fill=white] (3,3) circle (.3);
\node at (3,3) {$z_i$};
\path[fill=white] (3,1) circle (.3);
\node at (3,1) {$z_j$};
\path[fill=white] (1,2) circle (.3);
\node at (1,2) {$z_i,z_j$};
\end{tikzpicture}\qquad\qquad
\begin{tikzpicture}[baseline=(current bounding box.center)]
  \draw (0,1) to (2,1) to [out = 0, in = 180] (4,3);
  \draw (0,3) to (2,3) to [out = 0, in = 180] (4,1);
  \draw (1,0) to (1,4);
  \draw[fill=white] (0,1) circle (.3);
  \draw[fill=white] (0,3) circle (.3);
  \draw[fill=white] (1,4) circle (.3);
  \draw[fill=white] (4,3) circle (.3);
  \draw[fill=white] (4,1) circle (.3);
  \draw[fill=white] (1,0) circle (.3);
  \node at (0,1) {$a$};
  \node at (0,3) {$b$};
  \node at (1,4) {$c$};
  \node at (4,3) {$d$};
  \node at (4,1) {$e$};
  \node at (1,0) {$f$};
\path[fill=white] (1,3) circle (.3);
\node at (1,3) {$z_j$};
\path[fill=white] (1,1) circle (.3);
\node at (1,1) {$z_i$};
\path[fill=white] (3,2) circle (.3);
\node at (3,2) {$z_i,z_j$};
\end{tikzpicture}
\end{equation}
\end{proposition}

We note that Proposition~\ref{prop:uncolored_YBE} is an identity of $2^3 \times 2^3$ matrices, and so it is a finite computation to verify this still holds under the gauge transformation we have taken (see Remark~\ref{remark:gauge}).
Furthermore, we can see that the $R$-matrix corresponds to the vertices of the $L$-matrix rotated by $45^{\circ}$ clockwise and the weights of the $L$-matrix take $z = \frac{z_j - z_i}{1 + \beta z_i}$ and are scaled by $(1 + \beta z_i) z_j$.

\section{Colored lattice models and Lascoux atoms}\label{sec:coloredmodels}

We will build colored models that represent Lascoux atoms, generalizing the work in~\cite{BBBG19} where models were constructed for Demazure atoms.
The model we consider is a colored version of the lattice model of Motegi and Sakai~\cite{MS13} that represents a Grothendieck polynomial described in Section~\ref{sec:uncolored_model}.

\begin{figure}
\[
\begin{array}{c@{\;\;}c@{\;\;}c@{\;\;}c@{\;\;}c@{\;\;}c@{\;\;}c}
\toprule
  \tt{a}_1&\tt{a}_2&\tt{b}_1&\tt{b}^{\dagger}_1&\tt{b}^{\circ}_1&\tt{b}_2&\tt{c}_1\\
\midrule
\begin{tikzpicture}
\coordinate (a) at (-.75, 0);
\coordinate (b) at (0, .75);
\coordinate (c) at (.75, 0);
\coordinate (d) at (0, -.75);
\coordinate (aa) at (-.75,.5);
\coordinate (cc) at (.75,.5);
\draw (a)--(0,0);
\draw (b)--(0,0);
\draw (c)--(0,0);
\draw (d)--(0,0);
\draw[fill=white] (a) circle (.25);
\draw[fill=white] (b) circle (.25);
\draw[fill=white] (c) circle (.25);
\draw[fill=white] (d) circle (.25);
\node at (0,1) { };
\node at (a) {$0$};
\node at (b) {$0$};
\node at (c) {$0$};
\node at (d) {$0$};
\path[fill=white] (0,0) circle (.2);
\node at (0,0) {$z$};
\end{tikzpicture}
& \begin{tikzpicture}
\coordinate (a) at (-.75, 0);
\coordinate (b) at (0, .75);
\coordinate (c) at (.75, 0);
\coordinate (d) at (0, -.75);
\coordinate (aa) at (-.75,.5);
\coordinate (cc) at (.75,.5);
\draw (a)--(0,0);
\draw (b)--(0,0);
\draw[line width=0.5mm, brown] (c)--(0,0);
\draw[line width=0.5mm, brown] (d)--(0,0);
\draw[fill=white] (a) circle (.25);
\draw[fill=white] (b) circle (.25);
\draw[line width=0.5mm,brown, fill=white] (c) circle (.25);
\draw[line width=0.5mm,brown, fill=white] (d) circle (.25);
\node at (0,1) { };
\node at (a) {$0$};
\node at (b) {$0$};
\node at (c) {$d$};
\node at (d) {$d$};
\path[fill=white] (0,0) circle (.2);
\node at (0,0) {$z$};
\end{tikzpicture}
& \begin{tikzpicture}
\coordinate (a) at (-.75, 0);
\coordinate (b) at (0, .75);
\coordinate (c) at (.75, 0);
\coordinate (d) at (0, -.75);
\coordinate (aa) at (-.75,.5);
\coordinate (cc) at (.75,.5);
\draw[line width=0.5mm, blue] (a)--(0,0);
\draw[line width=0.5mm, blue] (b)--(0,0);
\draw[line width=0.5mm, red] (c)--(0,0);
\draw[line width=0.5mm, red] (d)--(0,0);
\draw[line width=0.5mm, blue,fill=white] (a) circle (.25);
\draw[line width=0.5mm, blue,fill=white] (b) circle (.25);
\draw[line width=0.5mm, red, fill=white] (c) circle (.25);
\draw[line width=0.5mm, red, fill=white] (d) circle (.25);
\node at (0,1) { };
\node at (a) {$c_j$};
\node at (b) {$c_j$};
\node at (c) {$c_i$};
\node at (d) {$c_i$};
\path[fill=white] (0,0) circle (.2);
\node at (0,0) {$z$};
\end{tikzpicture}
& \begin{tikzpicture}
\coordinate (a) at (-.75, 0);
\coordinate (b) at (0, .75);
\coordinate (c) at (.75, 0);
\coordinate (d) at (0, -.75);
\coordinate (aa) at (-.75,.5);
\coordinate (cc) at (.75,.5);
\draw[line width=0.5mm, blue] (a)--(0,0);
\draw[line width=0.6mm, red] (b)--(0,0);
\draw[line width=0.5mm, blue] (c)--(0,0);
\draw[line width=0.6mm, red] (d)--(0,0);
\draw[line width=0.5mm, blue,fill=white] (a) circle (.25);
\draw[line width=0.5mm, red,fill=white] (b) circle (.25);
\draw[line width=0.5mm, blue, fill=white] (c) circle (.25);
\draw[line width=0.5mm, red, fill=white] (d) circle (.25);
\node at (0,1) { };
\node at (a) {$c_j$};
\node at (b) {$c_i$};
\node at (c) {$c_j$};
\node at (d) {$c_i$};
\path[fill=white] (0,0) circle (.2);
\node at (0,0) {$z$};
\end{tikzpicture}
& \begin{tikzpicture}
\coordinate (a) at (-.75, 0);
\coordinate (b) at (0, .75);
\coordinate (c) at (.75, 0);
\coordinate (d) at (0, -.75);
\coordinate (aa) at (-.75,.5);
\coordinate (cc) at (.75,.5);
\draw[line width=0.5mm, brown] (a)--(0,0);
\draw[line width=0.6mm, brown] (b)--(0,0);
\draw[line width=0.5mm, brown] (c)--(0,0);
\draw[line width=0.6mm, brown] (d)--(0,0);
\draw[line width=0.5mm, brown,fill=white] (a) circle (.25);
\draw[line width=0.5mm, brown,fill=white] (b) circle (.25);
\draw[line width=0.5mm, brown, fill=white] (c) circle (.25);
\draw[line width=0.5mm, brown, fill=white] (d) circle (.25);
\node at (0,1) { };
\node at (a) {$d$};
\node at (b) {$d$};
\node at (c) {$d$};
\node at (d) {$d$};
\path[fill=white] (0,0) circle (.2);
\node at (0,0) {$z$};
\end{tikzpicture}
& \begin{tikzpicture}
\coordinate (a) at (-.75, 0);
\coordinate (b) at (0, .75);
\coordinate (c) at (.75, 0);
\coordinate (d) at (0, -.75);
\coordinate (aa) at (-.75,.5);
\coordinate (cc) at (.75,.5);
\draw[line width=0.5mm, brown] (a)--(0,0);
\draw(b)--(0,0);
\draw[line width=0.5mm, brown] (c)--(0,0);
\draw (d)--(0,0);
\draw[line width=0.5mm,brown,fill=white] (a) circle (.25);
\draw[fill=white] (b) circle (.25);
\draw[line width=0.5mm,brown,fill=white] (c) circle (.25);
\draw[fill=white] (d) circle (.25);
\node at (0,1) { };
\node at (a) {$d$};
\node at (b) {$0$};
\node at (c) {$d$};
\node at (d) {$0$};
\path[fill=white] (0,0) circle (.2);
\node at (0,0) {$z$};
\end{tikzpicture}
& \begin{tikzpicture}
\coordinate (a) at (-.75, 0);
\coordinate (b) at (0, .75);
\coordinate (c) at (.75, 0);
\coordinate (d) at (0, -.75);
\coordinate (aa) at (-.75,.5);
\coordinate (cc) at (.75,.5);
\draw[line width=0.5mm, brown] (a)--(0,0);
\draw[line width=0.6mm, brown] (b)--(0,0);
\draw (c)--(0,0);
\draw (d)--(0,0);
\draw[line width=0.5mm,brown,fill=white] (a) circle (.25);
\draw[line width=0.5mm,brown,fill=white] (b) circle (.25);
\draw[fill=white] (c) circle (.25);
\draw[fill=white] (d) circle (.25);
\node at (0,1) { };
\node at (a) {$d$};
\node at (b) {$d$};
\node at (c) {$0$};
\node at (d) {$0$};
\path[fill=white] (0,0) circle (.2);
\node at (0,0) {$z$};
\end{tikzpicture}
\\ 
\midrule
1 & 1+\beta z & 1 &1 & 1 & z & 1\\
\bottomrule
\end{array}
\]
\caption{The colored Boltzmann weights with ${\color{red} c_i} > {\color{blue} c_j}$ and ${\color{brown} d}$ being any color.}
\label{fig:colored_weights}
\end{figure}

Consider a rectangular grid of $n$ horizontal lines and $m$ vertical lines.
We also fix an $n$-tuple of colors $\cc = (c_1 > c_2 > \cdots > c_n > 0)$.
Let $w \in \sym{n}$, and let $w \cc = (c_{w(1)}, c_{w(2)}, \dotsc, c_{w(n)})$ be the colors permuted by $w$.
We label the edges by $\{0\} \sqcup \cc$ with the bottom and right (half) edges are labeled by $0$, the left (half) edges are labeled by $w w_0 \cc$ from top to bottom, and the top edges given by $\lambda$ with the $i$-th $1$ in the $\{0,1\}$-sequence of $\lambda$, counted from the left, having color $c_i$.
We give Boltzmann weights to the vertices according to Figure~\ref{fig:colored_weights}.
Let $\overline{\states}_{\lambda,w}$ denote the set of all possible admissible states for this model.

With the addition of colors and based on the admissible configurations, we can think of a state in $\states_{\lambda,w}$ as corresponding to a wiring diagram of $w$, where the different strands are represented by different colors.
Indeed, we can think of $\tt{a}_2$, $\tt{b}_2$, and $\tt{c}_1$ as a single strand passing through the vertex (possibly turning), $\tt{b}_1^{\dagger}$ as two strands crossing at the vertex (thus corresponding to a simple transposition), and $\tt{b}_1$ as two strands both passing near the vertex but not crossing.

\begin{remark}
In contrast to~\cite{BBBG19} we have colored the left side instead of the right side and we are working with $w w_0 \cc$ (instead of $w \cc$) due to our indexing convention.
Furthermore, we include this $w_0$ in order to emphasize that taking $\overline{\states}_{\lambda,1}$ corresponds to the wiring diagram of $w_0$ (with every pair of strands crossing exactly once).
\end{remark}

\begin{figure}
\[
\begin{array}{c@{\hspace{30pt}}c@{\hspace{30pt}}c@{\hspace{30pt}}c}
\toprule
\begin{tikzpicture}[scale=0.7]
\draw (0,0) to [out = 0, in = 180] (2,2);
\draw (0,2) to [out = 0, in = 180] (2,0);
\draw[fill=white] (0,0) circle (.35);
\draw[fill=white] (0,2) circle (.35);
\draw[fill=white] (2,0) circle (.35);
\draw[fill=white] (2,2) circle (.35);
\node at (0,0) {$0$};
\node at (0,2) {$0$};
\node at (2,2) {$0$};
\node at (2,0) {$0$};
\path[fill=white] (1,1) circle (.3);
\node at (1,1) {$z_i,z_j$};
\end{tikzpicture}&
\begin{tikzpicture}[scale=0.7]
\draw (0,0) to [out = 0, in = 180] (2,2);
\draw (0,2) to [out = 0, in = 180] (2,0);
\draw[fill=white] (0,0) circle (.35);
\draw[line width=0.5mm, brown, fill=white] (0,2) circle (.35);
\draw[line width=0.5mm, brown, fill=white] (2,2) circle (.35);
\draw[fill=white] (2,0) circle (.35);
\node at (0,0) {$0$};
\node at (0,2) {$d$};
\node at (2,2) {$d$};
\node at (2,0) {$0$};
\path[fill=white] (1,1) circle (.3);
\node at (1,1) {$z_i,z_j$};
\end{tikzpicture}&
\begin{tikzpicture}[scale=0.7]
\draw (0,0) to [out = 0, in = 180] (2,2);
\draw (0,2) to [out = 0, in = 180] (2,0);
\draw[fill=white] (0,0) circle (.35);
\draw[line width=0.5mm, brown, fill=white] (0,2) circle (.35);
\draw[fill=white] (2,2) circle (.35);
\draw[line width=0.5mm, brown, fill=white] (2,0) circle (.35);
\node at (0,0) {$0$};
\node at (0,2) {$d$};
\node at (2,2) {$0$};
\node at (2,0) {$d$};
\path[fill=white] (1,1) circle (.3);
\node at (1,1) {$z_i,z_j$};
\end{tikzpicture}&
\begin{tikzpicture}[scale=0.7]
\draw (0,0) to [out = 0, in = 180] (2,2);
\draw (0,2) to [out = 0, in = 180] (2,0);
\draw[line width=0.5mm, brown, fill=white] (0,0) circle (.35);
\draw[fill=white] (0,2) circle (.35);
\draw[fill=white] (2,2) circle (.35);
\draw[line width=0.5mm, brown, fill=white] (2,0) circle (.35);
\node at (0,0) {$d$};
\node at (0,2) {$0$};
\node at (2,2) {$0$};
\node at (2,0) {$d$};
\path[fill=white] (1,1) circle (.3);
\node at (1,1) {$z_i,z_j$};
\end{tikzpicture}\\
   \midrule
 (1 + \beta z_i) z_j & (1+ \beta z_i) z_j & (z_j - z_i) z_j & (1 + \beta z_j) z_j \\
   \midrule
\begin{tikzpicture}[scale=0.7]
\draw (0,0) to [out = 0, in = 180] (2,2);
\draw (0,2) to [out = 0, in = 180] (2,0);
\draw[line width=0.5mm, blue, fill=white] (0,0) circle (.35);
\draw[line width=0.5mm, red, fill=white] (0,2) circle (.35);
\draw[line width=0.5mm, red, fill=white] (2,2) circle (.35);
\draw[line width=0.5mm, blue, fill=white] (2,0) circle (.35);
\node at (0,0) {$c_j$};
\node at (0,2) {$c_i$};
\node at (2,2) {$c_i$};
\node at (2,0) {$c_j$};
\path[fill=white] (1,1) circle (.3);
\node at (1,1) {$z_i,z_j$};
\end{tikzpicture}&
\begin{tikzpicture}[scale=0.7]
\draw (0,0) to [out = 0, in = 180] (2,2);
\draw (0,2) to [out = 0, in = 180] (2,0);
\draw[line width=0.5mm, red, fill=white] (0,0) circle (.35);
\draw[line width=0.5mm, blue, fill=white] (0,2) circle (.35);
\draw[line width=0.5mm, blue, fill=white] (2,2) circle (.35);
\draw[line width=0.5mm, red, fill=white] (2,0) circle (.35);
\node at (0,0) {$c_i$};
\node at (0,2) {$c_j$};
\node at (2,2) {$c_j$};
\node at (2,0) {$c_i$};
\path[fill=white] (1,1) circle (.3);
\node at (1,1) {$z_i,z_j$};
\end{tikzpicture}&
\begin{tikzpicture}[scale=0.7]
\draw (0,0) to [out = 0, in = 180] (2,2);
\draw (0,2) to [out = 0, in = 180] (2,0);
\draw[line width=0.5mm, red, fill=white] (0,0) circle (.35);
\draw[line width=0.5mm, blue, fill=white] (0,2) circle (.35);
\draw[line width=0.5mm, red, fill=white] (2,2) circle (.35);
\draw[line width=0.5mm, blue, fill=white] (2,0) circle (.35);
\node at (0,0) {$c_i$};
\node at (0,2) {$c_j$};
\node at (2,2) {$c_i$};
\node at (2,0) {$c_j$};
\path[fill=white] (1,1) circle (.3);
\node at (1,1) {$z_i,z_j$};
\end{tikzpicture}&
\begin{tikzpicture}[scale=0.7]
\draw (0,0) to [out = 0, in = 180] (2,2);
\draw (0,2) to [out = 0, in = 180] (2,0);
\draw[line width=0.5mm, brown, fill=white] (0,0) circle (.35);
\draw[line width=0.5mm, brown, fill=white] (0,2) circle (.35);
\draw[line width=0.5mm, brown, fill=white] (2,2) circle (.35);
\draw[line width=0.5mm, brown, fill=white] (2,0) circle (.35);
\node at (0,0) {$d$};
\node at (0,2) {$d$};
\node at (2,2) {$d$};
\node at (2,0) {$d$};
\path[fill=white] (1,1) circle (.3);
\node at (1,1) {$z_i,z_j$};
\end{tikzpicture}\\
   \midrule
(1 + \beta z_j) z_ i & (1 + \beta z_i) z_j & z_j - z_i & (1 + \beta z_i) z_j \\
   \bottomrule
\end{array}
\]
\caption{The colored $R$-matrix with ${\color{red} c_i} > {\color{blue} c_j}$ and ${\color{brown} d}$ being any color. Note that the weights are not symmetric with respect to color.}
\label{fig:colored_R_matrix}
\end{figure}
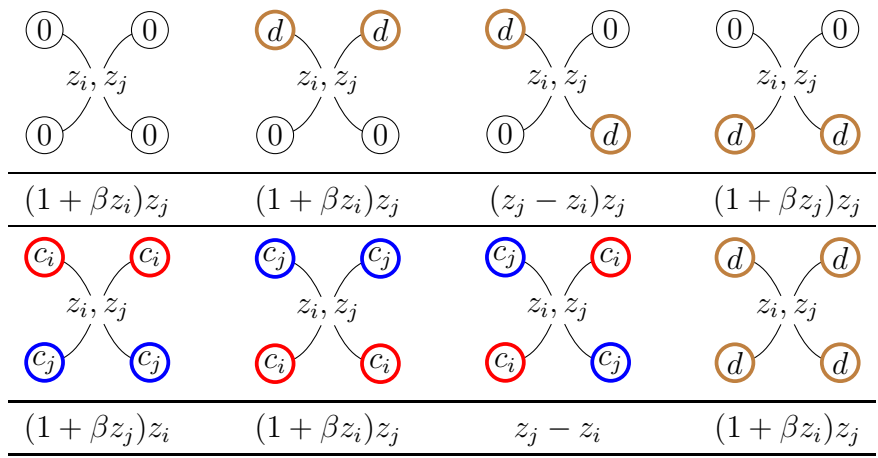

Our model is amenable to study via the Yang--Baxter equation.
We introduce the $R$-matrix for this model, which we call the \defn{colored $R$-matrix}, and the admissible configurations with their Boltzmann weights are given in Figure~\ref{fig:colored_R_matrix}.
To distinguish them from the usual vertices given by the $L$ matrix, we draw them tilted on their side.
Together with the previously introduced vertices, they satisfy the Yang--Baxter equation:

\begin{proposition}
\label{prop:YBE}
Consider the $L$-matrix given in Figure~\ref{fig:colored_weights} and $R$-matrix given in Figure~\ref{fig:colored_R_matrix}.
The partition function of the two models given by~\eqref{eq:RLL_relation} are equal for any boundary conditions $a,b,c,d,e,f \in \{ 0, c_1, \dotsc, c_n \}$. 
\end{proposition}

\begin{proof}
This is a computation that requires at most $3$ colors, so $R$ and $L$ are $4^3 \times 4^3$ matrices (note that the colors are conversed when applying the $R$-matrix and $L$-matrix).
Thus, it is a finite computation to check this that can easily be done by computer (for example, as given in Appendix~\ref{sec:sage_code}).
\end{proof}
 
By using the Yang--Baxter equation and the well known train argument, we can derive the following equation for the partition functions of our lattice model. 

\begin{lemma}
\label{lemma:recurrence_eq}
Let $w \in \sym{n}$, and consider $s_i$ be such that $s_i w > w$.
Then we have
\[
Z(\overline{\states}_{\lambda, s_i w}; \zz; \beta) = \frac{(1  + \beta z_i ) z_{i+1} \bigl( Z(\overline{\states}_{\lambda, w}; \zz; \beta) -  Z(\overline{\states}_{\lambda, w}; s_i \zz; \beta) \bigr)}{z_i - z_{i+1}}.
\]
\end{lemma}

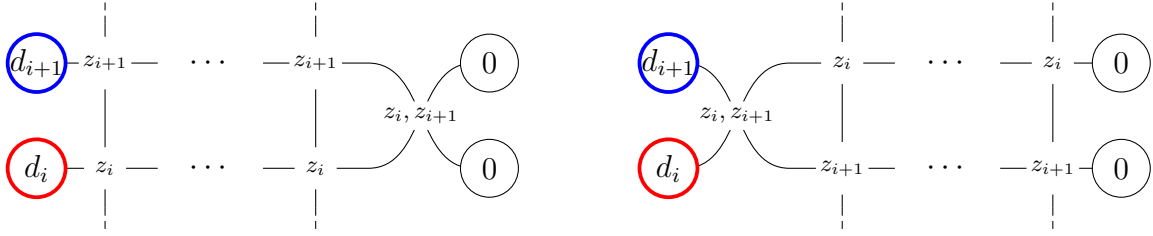
\begin{figure}
\begin{tikzpicture}[scale=0.7]
\begin{scope}[shift={(-6,0)}]
  \draw (4,1) to (6,1) to [out = 0, in = 180] (8,3);
  \draw (4,3) to (6,3) to [out = 0, in = 180] (8,1);
  \draw[fill=white] (8.3,3) circle (.55);
  \draw[fill=white] (8.3,1) circle (.55);
  \draw (0,1) to (2,1);
  \draw (0,3) to (2,3);
  \draw (5,0.25) to (5,3.75);
  \draw (1,0.25) to (1,3.75);
  \draw[line width=0.5mm,red,fill=white] (-0.3,1) circle (.55);
  \draw[line width=0.5mm,blue,fill=white] (-0.3,3) circle (.55);

  \node at (3,1) {$\cdots$};
  \node at (3,3) {$\cdots$};

  \draw[densely dashed] (1,3.75) to (1,4.25);
  \draw[densely dashed] (1,0.25) to (1,-0.25);
  \draw[densely dashed] (5,3.75) to (5,4.25);
  \draw[densely dashed] (5,0.25) to (5,-0.25);
  \path[fill=white] (1,3) circle (.5);
  \node at (1,3) {\scriptsize$z_{i+1}$};
  \path[fill=white] (1,1) circle (.4);
  \node at (1,1) {\scriptsize$z_i$};

  \path[fill=white] (5,3) circle (.5);
  \node (a) at (5,3) {\scriptsize$z_{i+1}$};
  \path[fill=white] (5,1) circle (.4);
  \node at (5,1) {\scriptsize$z_i$};

  \path[fill=white] (7,2) circle (.3);
  \node at (7,2) {\scriptsize$z_i,z_{i+1}$};
  \node at (8.3,1) {$0$};
  \node at (8.3,3) {$0$};
  \node at (-0.3,1) {$d_i$};
  \node at (-0.3,3) {$d_{i+1}$};
\end{scope}
\begin{scope}[shift={(6,0)}]
  \draw (0,1) to [out = 0, in = 180] (2,3) to (4,3);
  \draw (0,3) to [out = 0, in = 180] (2,1) to (4,1);
  \draw (3,0.25) to (3,3.75);
  \draw (7,0.25) to (7,3.75);
  \draw (6,1) to (8,1);
  \draw (6,3) to (8,3);
  \draw[line width=0.5mm,red,fill=white] (-0.3,1) circle (.55);
  \draw[line width=0.5mm,blue,fill=white] (-0.3,3) circle (.55);
  \draw[fill=white] (8.3,3) circle (.55);
  \draw[fill=white] (8.3,1) circle (.55);
  \node at (-0.3,1) {$d_i$};
  \node at (-0.3,3) {$d_{i+1}$};
  \node at (5,3) {$\cdots$};
  \node at (5,1) {$\cdots$};
  \draw[densely dashed] (3,3.75) to (3,4.25);
  \draw[densely dashed] (3,0.25) to (3,-0.25);
  \draw[densely dashed] (7,3.75) to (7,4.25);
  \draw[densely dashed] (7,0.25) to (7,-0.25);
  \node at (8.3,1) {$0$};
  \node at (8.3,3) {$0$};
\path[fill=white] (3,3) circle (.4);
\node at (3,3) {\scriptsize$z_i$};
\path[fill=white] (3,1) circle (.5);
\node at (3,1) {\scriptsize$z_{i+1}$};
\path[fill=white] (7,3) circle (.4);
\node at (7,3) {\scriptsize$z_i$};
\path[fill=white] (7,1) circle (.5);
\node at (7,1) {\scriptsize$z_{i+1}$};
\path[fill=white] (1,2) circle (.3);
\node at (1,2) {\scriptsize$z_i,z_{i+1}$};
\end{scope}
\end{tikzpicture}
\caption{Left: The model $\overline{\states}_{\lambda,w}$ with an $R$-matrix attached on the right. Right: The model after using the Yang--Baxter equation in the same model.}
\label{fig:train_argument}
\end{figure}

\begin{proof}
Consider the model $\overline{\states}_{\lambda,w}$, and let $\mathbf{d} = w w_0 \cc$.
Since $s_i w > w$, we note that $d_{i+1} < d_i$.
We construct a new model $\model$ by adding an $R$-matrix $R(z_i,z_{i+1})$ to the right with rightmost boundary entries being $0$.
Note that there is a bijection between states of $\model$ and $\overline{\states}_{\lambda,w}$ as there is precisely one admissible configuration at the $R$-matrix.
However, there is an additional factor due to the Boltzmann weight of the $R$-matrix configuration, and hence, we have $Z(\model; \zz; \beta) = (1 + \beta z_i) z_{i+1} Z(\overline{\states}_{\lambda,w}; s_i \zz; \beta)$.
Next, by repeatedly using the Yang--Baxter equation (Proposition~\ref{prop:YBE}), we obtain an equivalent model $\model'$ with an $R$-matrix on the left with colors $d_i$ and $d_{i+1}$.
For a pictorial description, see Figure~\ref{fig:train_argument}.
In this case, we have two possible admissible configurations for the $R$-matrix, one of which corresponds to $\overline{\states}_{\lambda,w}$ and the other to $\overline{\states}_{\lambda,s_i w}$.
Therefore, we obtain
\begin{align*}
(1 + \beta z_i) z_{i+1} Z(\overline{\states}_{\lambda,w}; s_i \zz; \beta) & = Z(\model; \zz; \beta) = Z(\model'; \zz; \beta)
\\ & = (1 + \beta z_i) z_{i+1} Z(\overline{\states}_{\lambda,w}; \zz; \beta) +  (z_{i+1} - z_i) Z(\overline{\states}_{\lambda,s_i w}; \zz; \beta).
\end{align*}
Solving this for $Z(\states_{\lambda,s_i w}; \zz; \beta)$, we obtain our desired formula.
\end{proof}

\begin{theorem}
\label{thm:5vertex_atoms}
We have
\[
\overline{L}_{w\lambda}(\zz; \beta) = Z(\overline{\states}_{\lambda,w}; \zz; \beta).
\]
\end{theorem}

\begin{proof}
A direct computation using the definition of $\overline{\varpi}_i$ yields
\[
\overline{L}_{s_i w \lambda}(\zz;\beta) = \overline{\varpi}_i \overline{L}_{w\lambda}(\zz;\beta) 
= \frac{(1  + \beta z_i ) z_{i+1} \cdot \bigl( \overline{L}_{w\lambda}(\zz;\beta) -  \overline{L}_{w\lambda}(s_i \zz; \beta) \bigr)}{z_i - z_{i+1}}.
\]
It is straightforward to see that $Z(\overline{\states}_{1,\lambda}; \zz; \beta) = \zz^{\lambda} = \overline{L}_{\lambda}(\zz; \beta)$, where $1 \in \sym{n}$ is the identity element.
Therefore, the claim follows by induction and Lemma~\ref{lemma:recurrence_eq}.
\end{proof}

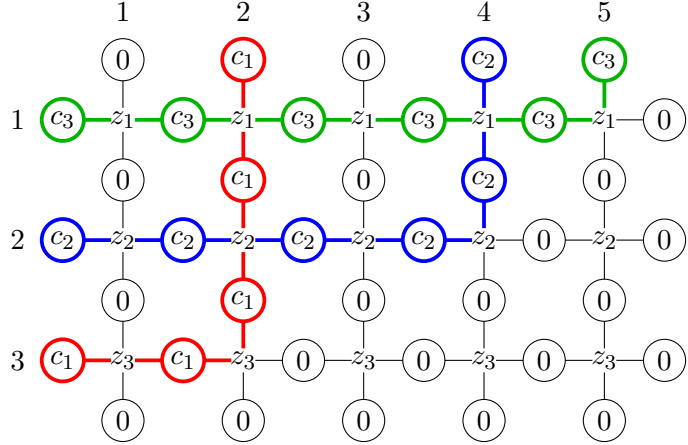
\begin{figure}[h]
  \begin{tikzpicture}[scale=0.80, font=\small]
    \foreach \y in {1,3,5}
        \draw (0,\y)--(10,\y);
    \foreach \x in {1,3,5,7,9}
      \draw (\x,0)--(\x,6);
    \draw[line width=0.5mm,darkgreen] (9,6) -- (9,5) -- (0,5);
    \draw[line width=0.5mm,blue] (7,6) -- (7,3) -- (0,3);
    \draw[line width=0.5mm,red] (3,6) -- (3,1) -- (0,1);
    \foreach \x in {0,2,4,6,8,10}
    {
      \draw[fill=white] (\x,5) circle (.35);
      \draw[fill=white] (\x,3) circle (.35);
      \draw[fill=white] (\x,1) circle (.35);
    }
    \foreach \x in {1,3,5,7,9}
    {
      \draw[fill=white] (\x,0) circle (.35);
      \draw[fill=white] (\x,2) circle (.35);
      \draw[fill=white] (\x,4) circle (.35);
      \draw[fill=white] (\x,6) circle (.35);
      \path[fill=white] (\x,1) circle (.2);
      \node at (\x,1) {$z_3$};
      \path[fill=white] (\x,3) circle (.2);
      \node at (\x,3) {$z_2$};
      \path[fill=white] (\x,5) circle (.2);
      \node at (\x,5) {$z_1$};
    }

    \node at (1,6) {$0$};
    \draw[line width=0.5mm,red,fill=white] (3,6) circle (.35); \node at (3,6) {$c_1$};
    \node at (5,6) {$0$};
    \draw[line width=0.5mm,blue,fill=white] (7,6) circle (.35); \node at (7,6) {$c_2$};
    \draw[line width=0.5mm,darkgreen,fill=white] (9,6) circle (.35); \node at (9,6) {$c_3$};
    \node at (1,4) {$0$};
    \draw[line width=0.5mm,red,fill=white] (3,4) circle (.35); \node at (3,4) {$c_1$};
    \node at (5,4) {$0$};
    \draw[line width=0.5mm,blue,fill=white] (7,4) circle (.35); \node at (7,4) {$c_2$};
    \node at (9,4) {$0$};
    \node at (1,2) {$0$};
    \draw[line width=0.5mm,red,fill=white] (3,2) circle (.35);, \node at (3,2) {$c_1$};
    \node at (5,2) {$0$};
    \node at (7,2) {$0$};
    \node at (9,2) {$0$};
    \node at (1,0) {$0$};
    \node at (3,0) {$0$};
    \node at (5,0) {$0$};
    \node at (7,0) {$0$};
    \node at (9,0) {$0$};
    \draw[line width=0.5mm,darkgreen,fill=white] (0,5) circle (.35); \node at (0,5) {$c_3$};
    \draw[line width=0.5mm,darkgreen,fill=white] (2,5) circle (.35); \node at (2,5) {$c_3$};
    \draw[line width=0.5mm,darkgreen,fill=white] (4,5) circle (.35); \node at (4,5) {$c_3$};
    \draw[line width=0.5mm,darkgreen,fill=white] (6,5) circle (.35); \node at (6,5) {$c_3$};
    \draw[line width=0.5mm,darkgreen,fill=white] (8,5) circle (.35); \node at (8,5) {$c_3$};
    \node at (10,5) {$0$};
    \draw[line width=0.5mm,blue,fill=white] (0,3) circle (.35); \node at (0,3) {$c_2$};
    \draw[line width=0.5mm,blue,fill=white] (2,3) circle (.35); \node at (2,3) {$c_2$};
    \draw[line width=0.5mm,blue,fill=white] (4,3) circle (.35); \node at (4,3) {$c_2$};
    \draw[line width=0.5mm,blue,fill=white] (6,3) circle (.35); \node at (6,3) {$c_2$};
    \node at (8,3) {$0$};
    \node at (10,3) {$0$};
    \draw[line width=0.5mm,red,fill=white] (0,1) circle (.35); \node at (0,1) {$c_1$};
    \draw[line width=0.5mm,red,fill=white] (2,1) circle (.35); \node at (2,1) {$c_1$};
    \node at (4,1) {$0$};
    \node at (6,1) {$0$};
    \node at (8,1) {$0$};
    \node at (10,1) {$0$};
    \node at (1.00,6.8) {$ 1$};
    \node at (3.00,6.8) {$ 2$};
    \node at (5.00,6.8) {$ 3$};
    \node at (7.00,6.8) {$ 4$};
    \node at (9.00,6.8) {$ 5$};
    \node at (-.75,1) {$ 3$};
    \node at (-.75,3) {$ 2$};
    \node at (-.75,5) {$ 1$};
  \end{tikzpicture}
  \caption{The unique ``ground'' state for the colored system $\overline{\states}_{\lambda,1}$ with $m =5$, $n=3$, and $\lambda = (2,2,1)$.
    We use colors ${\color{red} c_1} > {\color{blue} c_2} >  {\color{darkgreen} c_3}$.
    The Boltzmann weight of this state is $z_1^2z_2^2z_3$.}
  \label{fig:groundstate}
\end{figure}

\begin{figure}[h]
\begin{tikzpicture}[scale=0.80, font=\small]
    \foreach \y in {1,3,5}
        \draw (0,\y)--(16,\y);
    \foreach \x in {1,3,5,7,9,11,13,15}
      \draw (\x,0)--(\x,6);
    \draw[line width=0.5mm,darkgreen] (13,6) -- (13,5) -- (5,5) -- (5,3) -- (0,3);
    \draw[line width=0.5mm,blue] (7,6) -- (7,3) -- (5,3) -- (5,1) -- (0,1);
    \draw[line width=0.5mm,red] (3,6) -- (3,5) -- (0,5);
    \foreach \x in {0,2,4,6,8,10,12,14,16}
    {
      \draw[fill=white] (\x,5) circle (.35);
      \draw[fill=white] (\x,3) circle (.35);
      \draw[fill=white] (\x,1) circle (.35);
    }
    \foreach \x in {1,3,5,7,9,11,13,15}
    {
      \draw[fill=white] (\x,0) circle (.35);
      \draw[fill=white] (\x,2) circle (.35);
      \draw[fill=white] (\x,4) circle (.35);
      \draw[fill=white] (\x,6) circle (.35);
      \path[fill=white] (\x,1) circle (.2);
      \node at (\x,1) {$z_3$};
      \path[fill=white] (\x,3) circle (.2);
      \node at (\x,3) {$z_2$};
      \path[fill=white] (\x,5) circle (.2);
      \node at (\x,5) {$z_1$};
    }

    \node at (1,6) {$0$};
    \draw[line width=0.5mm,red,fill=white] (3,6) circle (.35); \node at (3,6) {$c_1$};
    \node at (5,6) {$0$};
    \draw[line width=0.5mm,blue,fill=white] (7,6) circle (.35); \node at (7,6) {$c_2$};
    \node at (9,6) {$0$};
    \node at (11,6) {$0$};
    \draw[line width=0.5mm,darkgreen,fill=white] (13,6) circle (.35); \node at (13,6) {$c_3$};
    \node at (15,6) {$0$};
    \node at (1,4) {$0$};
    \node at (3,4) {$0$};
    \draw[line width=0.5mm,darkgreen,fill=white] (5,4) circle (.35); \node at (5,4) {$c_3$};
    \draw[line width=0.5mm,blue,fill=white] (7,4) circle (.35); \node at (7,4) {$c_2$};
    \node at (9,4) {$0$};
    \node at (11,4) {$0$};
    \node at (13,4) {$0$};
    \node at (15,4) {$0$};
    \node at (1,2) {$0$};
    \node at (3,2) {$0$};
    \draw[line width=0.5mm,blue,fill=white] (5,2) circle (.35); \node at (5,2) {$c_2$};
    \node at (7,2) {$0$};
    \node at (9,2) {$0$};
    \node at (11,2) {$0$};
    \node at (13,2) {$0$};
    \node at (15,2) {$0$};
    \node at (1,0) {$0$};
    \node at (3,0) {$0$};
    \node at (5,0) {$0$};
    \node at (7,0) {$0$};
    \node at (9,0) {$0$};
    \node at (11,0) {$0$};
    \node at (13,0) {$0$};
    \node at (15,0) {$0$};
    \draw[line width=0.5mm,red,fill=white] (0,5) circle (.35); \node at (0,5) {$c_1$};
    \draw[line width=0.5mm,red,fill=white] (2,5) circle (.35); \node at (2,5) {$c_1$};
    \node at (4,5) {$0$};
    \draw[line width=0.5mm,darkgreen,fill=white] (6,5) circle (.35); \node at (6,5) {$c_3$};
    \draw[line width=0.5mm,darkgreen,fill=white] (8,5) circle (.35); \node at (8,5) {$c_3$};
    \draw[line width=0.5mm,darkgreen,fill=white] (10,5) circle (.35); \node at (10,5) {$c_3$};
    \draw[line width=0.5mm,darkgreen,fill=white] (12,5) circle (.35); \node at (12,5) {$c_3$};
    \node at (14,5) {$0$};
    \node at (16,5) {$0$};
    \draw[line width=0.5mm,darkgreen,fill=white] (0,3) circle (.35); \node at (0,3) {$c_3$};
    \draw[line width=0.5mm,darkgreen,fill=white] (2,3) circle (.35); \node at (2,3) {$c_3$};
    \draw[line width=0.5mm,darkgreen,fill=white] (4,3) circle (.35); \node at (4,3) {$c_3$};
    \draw[line width=0.5mm,blue,fill=white] (6,3) circle (.35); \node at (6,3) {$c_2$};
    \node at (8,3) {$0$};
    \node at (10,3) {$0$};
    \node at (12,3) {$0$};
    \node at (14,3) {$0$};
    \node at (16,3) {$0$};
    \draw[line width=0.5mm,blue,fill=white] (0,1) circle (.35); \node at (0,1) {$c_2$};
    \draw[line width=0.5mm,blue,fill=white] (2,1) circle (.35); \node at (2,1) {$c_2$};
    \draw[line width=0.5mm,blue,fill=white] (4,1) circle (.35); \node at (4,1) {$c_2$};
    \node at (6,1) {$0$};
    \node at (8,1) {$0$};
    \node at (10,1) {$0$};
    \node at (12,1) {$0$};
    \node at (14,1) {$0$};
    \node at (16,1) {$0$};
    \node at (1.00,6.8) {$ 1$};
    \node at (3.00,6.8) {$ 2$};
    \node at (5.00,6.8) {$ 3$};
    \node at (7.00,6.8) {$ 4$};
    \node at (9.00,6.8) {$ 5$};
    \node at (11.00,6.8) {$ 6$};
    \node at (13.00,6.8) {$ 7$};
    \node at (15.00,6.8) {$ 8$};
    \node at (-.75,1) {$ 3$};
    \node at (-.75,3) {$ 2$};
    \node at (-.75,5) {$ 1$};
  \end{tikzpicture}
  \caption{A state for the colored system $\overline{\states}_{\lambda,s_1 s_2}$, with $m =8$, $n=3$, and $\lambda = (4,2,1)$.
    We use colors ${\color{red} c_1} > {\color{blue} c_2} >  {\color{darkgreen} c_3}$.
    The Boltzmann weight of this state is $(1+\beta z_1)z_1^3 z_2^2 z_3^2$.}
  \label{fig:coloredstate}
\end{figure}
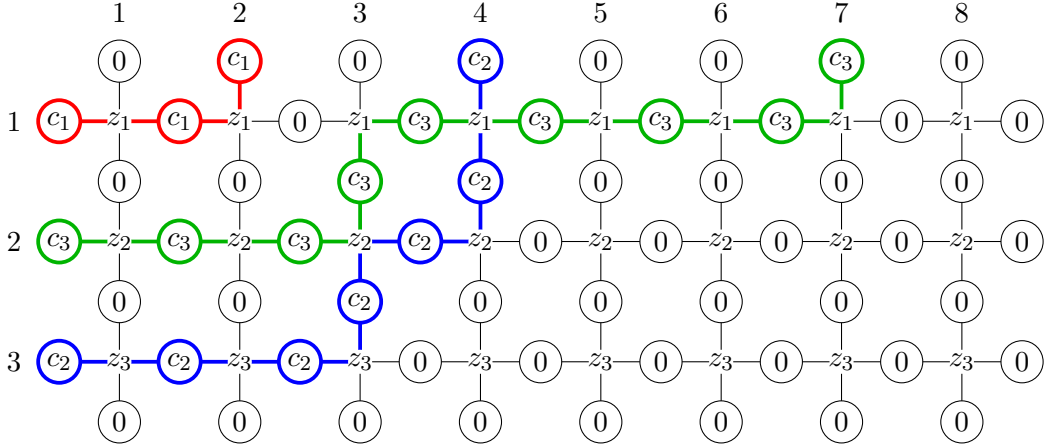

Let $\states_{\lambda,w}$ denote the same model as $\overline{\states}_{\lambda,w}$ with the additional colored configurations
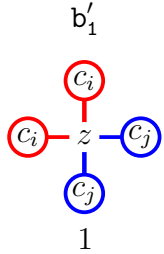
\begin{gather*}
\tt{b}'_1
\\
\begin{tikzpicture}[baseline=0]
\coordinate (a) at (-.75, 0);
\coordinate (b) at (0, .75);
\coordinate (c) at (.75, 0);
\coordinate (d) at (0, -.75);
\coordinate (aa) at (-.75,.5);
\coordinate (cc) at (.75,.5);
\draw[line width=0.5mm, red] (a)--(0,0);
\draw[line width=0.6mm, red] (b)--(0,0);
\draw[line width=0.5mm, blue] (c)--(0,0);
\draw[line width=0.6mm, blue] (d)--(0,0);
\draw[line width=0.5mm, red,fill=white] (a) circle (.25);
\draw[line width=0.5mm, red,fill=white] (b) circle (.25);
\draw[line width=0.5mm, blue, fill=white] (c) circle (.25);
\draw[line width=0.5mm, blue, fill=white] (d) circle (.25);
\node at (0,1) { };
\node at (a) {$c_i$};
\node at (b) {$c_i$};
\node at (c) {$c_j$};
\node at (d) {$c_j$};
\path[fill=white] (0,0) circle (.2);
\node at (0,0) {$z$};
\end{tikzpicture}
\\
1
\end{gather*}
whenever the colors ${\color{red} c_i} > {\color{blue} c_j}$ do not cross in $\overline{\states}_{\lambda,w}$; \textit{i.e.}, $(i, j)$ is not an inversion of $w w_0$.
In this case, we note that we can remove configurations $\tt{b}^{\dagger}$ and $\tt{b}_1$ for colors ${\color{red} c_i} > {\color{blue} c_j}$ from the model without changing the possible states.
We extend the definition of the $R$-matrix by using those in Figure~\ref{fig:colored_R_matrix} except we require the bottom left two configurations to have colors ${\color{red} c_i}$ and ${\color{blue} c_j}$ cross and we add the additional two configurations for when they do not cross:
\begin{equation}
\label{eq:noncrossing_weights}
\begin{array}{c@{\hspace{80pt}}c}
\begin{tikzpicture}[scale=0.7]
\draw (0,0) to [out = 0, in = 180] (2,2);
\draw (0,2) to [out = 0, in = 180] (2,0);
\draw[line width=0.5mm, blue, fill=white] (0,0) circle (.35);
\draw[line width=0.5mm, red, fill=white] (0,2) circle (.35);
\draw[line width=0.5mm, red, fill=white] (2,2) circle (.35);
\draw[line width=0.5mm, blue, fill=white] (2,0) circle (.35);
\node at (0,0) {$c_j$};
\node at (0,2) {$c_i$};
\node at (2,2) {$c_i$};
\node at (2,0) {$c_j$};
\path[fill=white] (1,1) circle (.3);
\node at (1,1) {$z_i,z_j$};
\end{tikzpicture}
&
\begin{tikzpicture}[scale=0.7]
\draw (0,0) to [out = 0, in = 180] (2,2);
\draw (0,2) to [out = 0, in = 180] (2,0);
\draw[line width=0.5mm, red, fill=white] (0,0) circle (.35);
\draw[line width=0.5mm, blue, fill=white] (0,2) circle (.35);
\draw[line width=0.5mm, blue, fill=white] (2,2) circle (.35);
\draw[line width=0.5mm, red, fill=white] (2,0) circle (.35);
\node at (0,0) {$c_i$};
\node at (0,2) {$c_j$};
\node at (2,2) {$c_j$};
\node at (2,0) {$c_i$};
\path[fill=white] (1,1) circle (.3);
\node at (1,1) {$z_i,z_j$};
\end{tikzpicture} \\
(1 + \beta z_i) z_j & (1 + \beta z_j) z_ i
\end{array}
\end{equation}
(We have simply interchanged the bottom left two Boltzmann weights from the $R$-matrix in Figure~\ref{fig:colored_R_matrix}.)
Indeed, this allows us to show the model is integrable.

\begin{proposition}
\label{prop:Lascoux_YBE}
Consider the modified $L$-matrix and $R$-matrix given above for $\states_{\lambda,w}$.
The partition function of the following two models given by~\eqref{eq:RLL_relation} are equal for any boundary conditions $a,b,c,d,e,f \in \{ 0, c_1, \dotsc, c_n \}$ and valid crossings coming from a wiring diagram.
\end{proposition}

\begin{proof}
The proof is similar to the proof of Proposition~\ref{prop:YBE} except we go over all possible valid crossings of colors ${\color{red}c_i}$ and ${\color{blue}c_j}$ and notice that the resulting $R$-matrix agrees for all such valid crossings.
Note that we cannot have color $c_1$ and $c_3$ cross without either $c_1$ and $c_2$ crossing or $c_2$ and $c_3$ crossing, which can be observed by looking at the $6$ permutations of $\sym{3}$.
\end{proof}

\begin{theorem}
\label{thm:5vertex_Lascoux}
We have
\[
L_{w\lambda}(\zz; \beta) = Z(\states_{\lambda,w}; \zz; \beta).
\]
\end{theorem}

We now give two proofs of this result.
The first is using the train argument as we used to prove Theorem~\ref{thm:5vertex_atoms}, and the second is combinatorial and applying Theorem~\ref{thm:5vertex_atoms}.

\begin{proof}[Proof using the train argument]
We have $Z(\states_{\lambda, s_i w}; \zz; \beta) = \varpi_i Z(\states_{\lambda, w}; \zz; \beta)$ since
\[
Z(\states_{\lambda, s_i w}; \zz; \beta) = \frac{(1  + \beta z_{i+1}) z_i Z(\states_{\lambda, w}; \zz; \beta) -  (1 + \beta z_i) z_{i+1} Z(\states_{\lambda, w}; s_i \zz; \beta) \bigr)}{z_i - z_{i+1}}
\]
as in the proof of Lemma~\ref{lemma:recurrence_eq} with
\begin{align*}
(1 + \beta z_i) z_{i+1} Z(\states_{\lambda,w}; s_i \zz; \beta) & = Z(\model; \zz; \beta) = Z(\model; \zz; \beta)
\\ & = (1 + \beta z_{i+1}) z_i Z(\states_{\lambda,w}; \zz; \beta) +  (z_{i+1} - z_i) Z(\states_{\lambda,s_i w}; \zz; \beta)
\end{align*}
and noting $c_i$ and $c_{i+1}$ do not cross as $\ell(s_i w) = \ell(w) + 1$.
\end{proof}

\begin{proof}[Combinatorial proof]
Replacing any particular configuration $\tt{b}'_1$ for colors ${\color{red}c_i}$ and ${\color{blue}c_j}$ in a state of $\states_{\lambda,w}$ with the configuration $\tt{b}^{\dagger}_1$ such that all resulting states are valid in $\states_{\lambda,u}$ corresponds to having $u < w$ (note that we necessarily have $\ell(w) = \ell(u) + 1$ and this is will be a cover in Bruhat order).
We can also see this by considering this as adding a crossing from the corresponding wiring diagram of $w w_0 < u w_0$ and that multiplying by $w_0$ is equivalent to taking the dual of Bruhat order (see, \textit{e.g.},~\cite[Prop.~2.3.2]{BB05}).
This swap does not result in an invalid state since the colors do not cross (\textit{i.e.}, there is no $\tt{b}^{\dagger}$ for these colors).
If we remove the northeast most $\tt{b}'_1$ for colors ${\color{red}c_i}$ and ${\color{blue}c_j}$ and all other touch points between the colors becomes $\tt{b}_1$, and note that there are no other $\tt{b}'_1$ for colors ${\color{red}c_i}$ and ${\color{blue}c_j}$ in the resulting state.
To see the reverse containment, we note that there might be a third color ${\color{darkgreen}c_k}$ that would be forced to cross one either ${\color{red}c_i}$ or ${\color{blue}c_j}$ twice after the swap, but we can resolve the crossings to obtain a valid state as before.
Therefore, the claim follows from Theorem~\ref{thm:5vertex_atoms},
\[
L_{w\lambda}(\zz; \beta) = \sum_{u \leq w} \overline{L}_{u\lambda}(\zz; \beta),
\]
and a straightforward induction on the length of $w$.
\end{proof}

Note that we could use the train argument proof and then use the combinatorial proof to show
\[
L_{w\lambda}(\zz; \beta) = \sum_{u \leq w} \overline{L}_{u\lambda}(\zz; \beta)
\]
as a consequence (thus yielding an alternative proof of~\cite[Thm.~5.1]{Monical16}).

\begin{example}
We consider replacing the $\tt{b}^{\dagger}$ configuration corresponding to ${\color{red}c_1}$ and ${\color{darkgreen}c_3}$ in $\states_{\emptyset,1}$ for $m = n = 3$ with colors ${\color{red}c_1} > {\color{blue}c_2} > {\color{darkgreen}c_3}$.
This introduces a double crossing of the colors ${\color{blue}c_2}$ and ${\color{darkgreen}c_3}$.
We can then resolve this to a valid state in $\states_{\emptyset,w_0}$ by
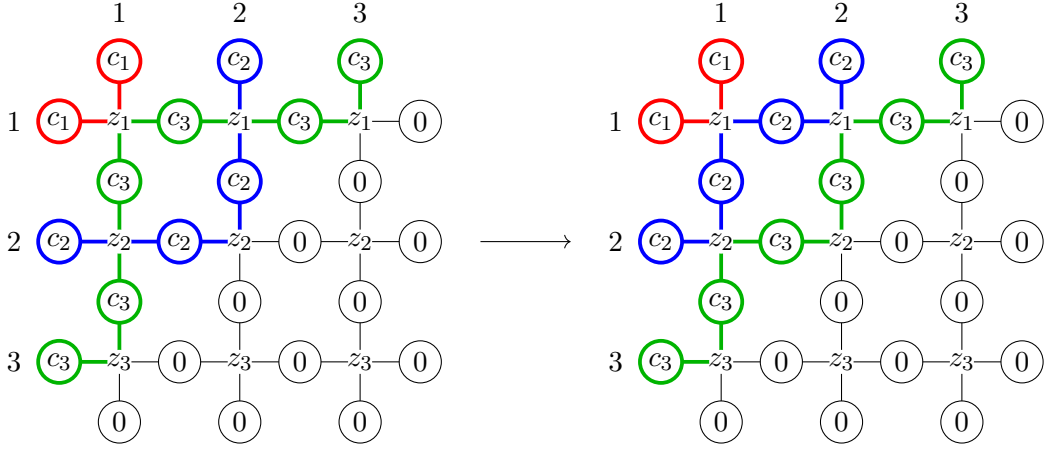
\[
  \begin{tikzpicture}[scale=0.80, font=\small,baseline=0]
    \foreach \y in {1,3,5}
        \draw (0,\y)--(6,\y);
    \foreach \x in {1,3,5}
      \draw (\x,0)--(\x,6);
    \draw[line width=0.5mm,darkgreen] (5,6) -- (5,5) -- (1,5) -- (1,1) -- (0,1);
    \draw[line width=0.5mm,blue] (3,6) -- (3,3) -- (0,3);
    \draw[line width=0.5mm,red] (1,6) -- (1,5) -- (0,5);
    \foreach \x in {0,2,4,6}
    {
      \draw[fill=white] (\x,5) circle (.35);
      \draw[fill=white] (\x,3) circle (.35);
      \draw[fill=white] (\x,1) circle (.35);
    }
    \foreach \x in {1,3,5}
    {
      \draw[fill=white] (\x,0) circle (.35);
      \foreach \y in {1,2,3} {
        \draw[fill=white] (\x,2*\y) circle (.35);
        \path[fill=white] (\x,7-2*\y) circle (.2);
        \node at (\x,7-2*\y) {$z_{\y}$};
      }
    }

    \draw[line width=0.5mm,red,fill=white] (1,6) circle (.35); \node at (1,6) {$c_1$};
    \draw[line width=0.5mm,blue,fill=white] (3,6) circle (.35); \node at (3,6) {$c_2$};
    \draw[line width=0.5mm,darkgreen,fill=white] (5,6) circle (.35); \node at (5,6) {$c_3$};
    \draw[line width=0.5mm,darkgreen,fill=white] (1,4) circle (.35); \node at (1,4) {$c_3$};
    \draw[line width=0.5mm,blue,fill=white] (3,4) circle (.35); \node at (3,4) {$c_2$};
    \node at (5,4) {$0$};
    \draw[line width=0.5mm,darkgreen,fill=white] (1,2) circle (.35);, \node at (1,2) {$c_3$};
    \node at (3,2) {$0$};
    \node at (5,2) {$0$};
    \node at (1,0) {$0$};
    \node at (3,0) {$0$};
    \node at (5,0) {$0$};
    \draw[line width=0.5mm,red,fill=white] (0,5) circle (.35); \node at (0,5) {$c_1$};
    \draw[line width=0.5mm,darkgreen,fill=white] (2,5) circle (.35); \node at (2,5) {$c_3$};
    \draw[line width=0.5mm,darkgreen,fill=white] (4,5) circle (.35); \node at (4,5) {$c_3$};
    \node at (6,5) {$0$};
    \draw[line width=0.5mm,blue,fill=white] (0,3) circle (.35); \node at (0,3) {$c_2$};
    \draw[line width=0.5mm,blue,fill=white] (2,3) circle (.35); \node at (2,3) {$c_2$};
    \node at (4,3) {$0$};
    \node at (6,3) {$0$};
    \draw[line width=0.5mm,darkgreen,fill=white] (0,1) circle (.35); \node at (0,1) {$c_3$};
    \node at (2,1) {$0$};
    \node at (4,1) {$0$};
    \node at (6,1) {$0$};
    \node at (1.00,6.8) {$ 1$};
    \node at (3.00,6.8) {$ 2$};
    \node at (5.00,6.8) {$ 3$};
    \node at (-.75,1) {$ 3$};
    \node at (-.75,3) {$ 2$};
    \node at (-.75,5) {$ 1$};
    
    \draw[->] (7,3) -- (8.5,3);
    
    \begin{scope}[xshift=10cm]
    \foreach \y in {1,3,5}
        \draw (0,\y)--(6,\y);
    \foreach \x in {1,3,5}
      \draw (\x,0)--(\x,6);
    \draw[line width=0.5mm,darkgreen] (5,6) -- (5,5) -- (3,5) -- (3,3) -- (1,3) -- (1,1) -- (0,1);
    \draw[line width=0.5mm,blue] (3,6) -- (3,5) -- (1,5) -- (1,3) -- (0,3);
    \draw[line width=0.5mm,red] (1,6) -- (1,5) -- (0,5);
    \foreach \x in {0,2,4,6}
    {
      \draw[fill=white] (\x,5) circle (.35);
      \draw[fill=white] (\x,3) circle (.35);
      \draw[fill=white] (\x,1) circle (.35);
    }
    \foreach \x in {1,3,5}
    {
      \draw[fill=white] (\x,0) circle (.35);
      \foreach \y in {1,2,3} {
        \draw[fill=white] (\x,2*\y) circle (.35);
        \path[fill=white] (\x,7-2*\y) circle (.2);
        \node at (\x,7-2*\y) {$z_{\y}$};
      }
    }

    \draw[line width=0.5mm,red,fill=white] (1,6) circle (.35); \node at (1,6) {$c_1$};
    \draw[line width=0.5mm,blue,fill=white] (3,6) circle (.35); \node at (3,6) {$c_2$};
    \draw[line width=0.5mm,darkgreen,fill=white] (5,6) circle (.35); \node at (5,6) {$c_3$};
    \draw[line width=0.5mm,blue,fill=white] (1,4) circle (.35); \node at (1,4) {$c_2$};
    \draw[line width=0.5mm,darkgreen,fill=white] (3,4) circle (.35); \node at (3,4) {$c_3$};
    \node at (5,4) {$0$};
    \draw[line width=0.5mm,darkgreen,fill=white] (1,2) circle (.35);, \node at (1,2) {$c_3$};
    \node at (3,2) {$0$};
    \node at (5,2) {$0$};
    \node at (1,0) {$0$};
    \node at (3,0) {$0$};
    \node at (5,0) {$0$};
    \draw[line width=0.5mm,red,fill=white] (0,5) circle (.35); \node at (0,5) {$c_1$};
    \draw[line width=0.5mm,blue,fill=white] (2,5) circle (.35); \node at (2,5) {$c_2$};
    \draw[line width=0.5mm,darkgreen,fill=white] (4,5) circle (.35); \node at (4,5) {$c_3$};
    \node at (6,5) {$0$};
    \draw[line width=0.5mm,blue,fill=white] (0,3) circle (.35); \node at (0,3) {$c_2$};
    \draw[line width=0.5mm,darkgreen,fill=white] (2,3) circle (.35); \node at (2,3) {$c_3$};
    \node at (4,3) {$0$};
    \node at (6,3) {$0$};
    \draw[line width=0.5mm,darkgreen,fill=white] (0,1) circle (.35); \node at (0,1) {$c_3$};
    \node at (2,1) {$0$};
    \node at (4,1) {$0$};
    \node at (6,1) {$0$};
    \node at (1.00,6.8) {$ 1$};
    \node at (3.00,6.8) {$ 2$};
    \node at (5.00,6.8) {$ 3$};
    \node at (-.75,1) {$ 3$};
    \node at (-.75,3) {$ 2$};
    \node at (-.75,5) {$ 1$};
  \end{scope}
  \end{tikzpicture}
\]
\end{example}

Finally, we construct another variation on the model $\overline{\states}_{\lambda,w}$ where instead of adding $\tt{b}'_1$ for certain colors (and removing the corresponding $\tt{b}_1$ and $\tt{b}^{\dagger}$), we replace $\tt{b}_1$ with $\tt{b}'_1$ for all colors.
Let $\states'_{\lambda,w}$ denote this modified model.
We also use the following $R$-matrix given by Figure~\ref{fig:colored_R_matrix} except we \emph{replace} the two bottom left configurations by Equation~\eqref{eq:noncrossing_weights}.
This satisfies the Yang--Baxter equation (the proof is the same as Proposition~\ref{prop:YBE}):

\begin{proposition}
\label{prop:Lascoux_prime_YBE}
Consider the modified $L$-matrix and $R$-matrix given above for $\states'_{\lambda,w}$.
The partition function of the following two models given by~\eqref{eq:RLL_relation} are equal for any boundary conditions $a,b,c,d,e,f \in \{ 0, c_1, \dotsc, c_n \}$.
\end{proposition}

\begin{theorem}
\label{thm:modified2}
We have
\[
L_{w\lambda}(\zz; \beta) = Z(\states'_{\lambda,w}; \zz; \beta).
\]
\end{theorem}

\begin{proof}
We have $Z(\states'_{\lambda, s_i w}; \zz; \beta) = \varpi_i Z(\states'_{\lambda, w}; \zz; \beta)$ since
\[
Z(\states'_{\lambda, s_i w}; \zz; \beta) = \frac{(1  + \beta z_{i+1}) z_i Z(\states'_{\lambda, w}; \zz; \beta) -  (1 + \beta z_i) z_{i+1} Z(\states'_{\lambda, w}; s_i \zz; \beta) \bigr)}{z_i - z_{i+1}}
\]
as in the proof of Lemma~\ref{lemma:recurrence_eq} with
\begin{align*}
(1 + \beta z_i) z_{i+1} Z(\states'_{\lambda,w}; s_i \zz; \beta) & = Z(\model; \zz; \beta) = Z(\model'; \zz; \beta)
\\ & = (1 + \beta z_{i+1}) z_i Z(\states'_{\lambda,w}; \zz; \beta) +  (z_{i+1} - z_i) Z(\states'_{\lambda,s_i w}; \zz; \beta).
\end{align*}
Thus, the claim follows by induction.
\end{proof}

\begin{proposition}
There exists a weight-preserving bijection $\xi \colon \states'_{\lambda,w} \to \states_{\lambda,w}$.
\end{proposition}

\begin{proof}
Define $\xi$ by taking an state in $\states'_{\lambda,w}$ and for all colors ${\color{red} c_i}$ and ${\color{blue} c_j}$ such that the colors cross (\textit{i.e.}, there is at least one vertex $\tt{b}^{\dagger}$ and possibly $\tt{b}'_1$), replace the northeast $\tt{b}'_1$ with a $\tt{b}^{\dagger}$ and all other $\tt{b}'_1$ and $\tt{b}^{\dagger}$ with $\tt{b}_1$.
It is straightforward to see that $\xi$ is a bijection as the last point of contact between the strands colored ${\color{red} c_i}$ and ${\color{blue} c_j}$ must be a $\tt{b}^{\dagger}$ for any state in $\states'_{\lambda,w}$.
\end{proof}

We end this section with some remarks on our Boltzmann weights.
Our (colored) $R$- and $L$-matrices are $\beta$-generalizations of the (colored) $R$- and $L$-matrices in~\cite{BBBG19}.
They are different to the $R$- and $L$-matrices in~\cite{BBBGIwahori} and cannot be realized as a specialization of the $R$- and $L$-matrices of~\cite{BorodinWheelerColored}.
To see this, note that in our setting the configuration 
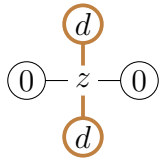
\begin{gather*}
\begin{tikzpicture}[baseline=0]
\coordinate (a) at (-.75, 0);
\coordinate (b) at (0, .75);
\coordinate (c) at (.75, 0);
\coordinate (d) at (0, -.75);
\coordinate (aa) at (-.75,.5);
\coordinate (cc) at (.75,.5);
\draw[] (a)--(0,0);
\draw[line width=0.5mm, brown] (b)--(0,0);
\draw[] (c)--(0,0);
\draw[line width=0.5mm, brown] (d)--(0,0);
\draw[ ,fill=white] (a) circle (.25);
\draw[line width=0.5mm, brown,fill=white] (b) circle (.25);
\draw[, fill=white] (c) circle (.25);
\draw[line width=0.5mm, brown, fill=white] (d) circle (.25);
\node at (0,1) { };
\node at (a) {$0$};
\node at (b) {$d$};
\node at (c) {$0$};
\node at (d) {$d$};
\path[fill=white] (0,0) circle (.2);
\node at (0,0) {$z$};
\end{tikzpicture}
\end{gather*}
is not admissible (\textit{i.e.}, it has Boltzmann weight equal to $0$), while in~\cite[Fig.~10]{BBBGIwahori} and~\cite[Fig.~2.2.2, Fig.~2.2.6]{BorodinWheelerColored}, this particular configuration is admissible and its weight cannot be specialized to $0$ for all $z$ unless we take $s \to \infty$ and $q \to 0$. These limits would then not allow us to match the weight of our $\tt{a}_2$ configuration with the one in~\cite{BorodinWheelerColored}.  
The reason for the difference is that going from ~\cite{BBBGIwahori} to the present work, we set their $v \to 0$ and then do a $\beta$ deformation which lands outside of the setting described in~\cite{BorodinWheelerColored}. 
This also holds even for the special case of $\beta = -1$.

Both the lattice model we introduce in this paper and the polynomials we study bear resemblance to the higher spin $U_q(\widehat{\mathfrak{sl}}_2)$ model and symmetric functions studied by Borodin and Petrov~\cite{Borodin17,BorodinPetrov18}.
However we are not able to establish a direct connection between the two settings.
This statement is also echoed in~\cite[Intro., Sec.~8.4]{Borodin17}, where Borodin states he does not know of a direct link between his work and that of Motegi and Sakai~\cite{MS13}.
In particular, Borodin gives a formula for his polynomials as a ratio of determinants that is similar to but different from that of Grothendieck polynomials (see Equation~\eqref{eq:determinant_formula} and~\cite{IN13}).

It would be interesting to understand the quantum group associated to our $R$- and $L$- matrices.
The quantum groups associated to the lattice models in~\cite{BorodinWheelerColored} and~\cite{BBBGIwahori} are $U_q(\widehat{\mathfrak{sl}}_n)$ and $U_q\bigl(\widehat{\mathfrak{gl}}(n|1)\bigr)$, respectively.
In both cases it is quite easy to see that the $R$-matrix is related to the standard representation of the corresponding quantum group.
This is in contrast to our case, where we were unable to identify our $R$-matrix as the $R$-matrix associated to a known quantum group (or with a Drinfeld twist).

\section{Lascoux atoms to Key tableaux and set-valued skyline tableaux}
\label{sec:proofofconjecture}

In this section, we prove Conjecture~\ref{conj:svt_Lascoux} and Conjecture~\ref{conj:skyline_tableaux}.

We refine the states of $\overline{\states}_{\lambda,w}$ to allow markings at the configurations $\tt{a}_2$.
More precisely, a \defn{marked state} is a pair $(S, M)$ with $S \in \overline{\states}_{\lambda,w}$ such that $M$ is some subset of all configurations of $\tt{a}_2$.
We note that $\fP$ naturally extends to a bijection between marked states $\bigsqcup_{w \in \sym{n}} \overline{\states}_{\lambda,w}$ and marked GT patterns with top row $\lambda$ as each configuration $\tt{a}_2$ in a state $S$ corresponds to a position where a marking is possible in $\fP(S)$.
As before, the weight gets twisted by $w_0$.
Thus, for any state $S \in \overline{\states}_{\lambda,w}$, we have
\begin{equation}
\label{eq:marked_weight}
\wt(S) = \sum_{(S, M)} \beta^{|M|} \wt(S, M) = \sum_{(\fP(S), M)} \beta^{|M|} w_0 \wt(\fP(S), M),
\end{equation}
where the first sum is over all possible markings of $S$ and the second sum is over all possible markings of $\fP(S)$ (see also Equation~\eqref{eq:wt_by_GT_patterns}).
We also can do the same refinement for $\states_{\lambda,w}$.

\begin{theorem}
\label{thm:key_theorem}
Conjecture~\ref{conj:svt_Lascoux} is true, that is to say
\[
\overline{L}_{w\lambda}(\zz; \beta) = \sum_{\substack{T \in \svt^n(\lambda) \\ K(T) = K_{w\lambda}}} \beta^{\excess(T)} \zz^{\wt(T)},
\qquad\qquad
L_{w\lambda}(\zz; \beta) = \sum_{\substack{T \in \svt^n(\lambda) \\ K(T) \leq K_{w\lambda}}} \beta^{\excess(T)} \zz^{\wt(T)}.
\]
\end{theorem}

\begin{proof}
We note that the Lusztig involution provides a weight preserving bijection between marked GT patterns for $\overline{\states}_{\lambda,w}$ and marked states of $\overline{\states}_{\lambda,w}$.
Forgetting all markings in a marked GT pattern is equivalent to taking the minimum entry in the corresponding set-valued tableau.
Thus, in order to describe the action of forgetting the markings on the state and be weight preserving, we are required to conjugate by the Lusztig involution.
Since the key tableau and the Lascoux atom is computed based on the unmarked state (Theorem~\ref{thm:5vertex_atoms}), the first claim follows from Equation~\eqref{eq:marked_weight}.

The second claim can be shown from the first and Equation~\eqref{eq:Lascoux_into_atoms} or directly by a similar argument as above with $\states_{\lambda,w}$ and Theorem~\ref{thm:5vertex_Lascoux}.
\end{proof}

\begin{example}
Let us examine how the proof of Theorem~\ref{thm:key_theorem} works on a particular example.
Consider the following set-valued tableau of shape $\lambda = (4,2,1)$ and its image under the Lusztig involution for $n = 3$:
\[
T = \ytableaushort{1111,2{2,\!3},3}\,,
\qquad\qquad
T^* = \ytableaushort{1133,2{2,\!3},3}\,.
\]
Thus the image of $T^*$ under $\phi$ is the state given by Figure~\ref{fig:coloredstate} with the unique $\tt{a}_2$ vertex being marked.
Then, we have
\[
\min(T^*) = \ytableaushort{1133,22,3}\,,
\]
which is the corresponding unmarked state under $\fP^{-1} \circ \phi^{-1}$. Applying the Lusztig involution again and taking the corresponding key tableau, we obtain
\[
\min(T^*)^* = \ytableaushort{1112,23,3}\,,
\qquad\qquad
k\bigl( \min(T^*)^* \bigr) = K(T) = \ytableaushort{1222,23,3}\,.
\]
We have $\wt\bigl(K(T) \bigr) = (1,4,2) = s_1 s_2 \lambda$, which agrees with $(\fP^{-1} \circ \phi^{-1})(T^*) \in \overline{\states}_{\lambda, s_1 s_2}$.
\end{example}

\begin{example}
Let $\lambda = (4, 2, 1)$.
We note there are two states in addition to the one in Figure~\ref{fig:coloredstate} in $\overline{\states}_{\lambda,s_1 s_2}$ with Boltzmann weights $(1 + \beta z_1)^2 z_1^2 z_2^3 z_3^2$ and $(1 + \beta z_1)^2 z_1 z_2^4 z_3^2$.
By applying $\phi \circ \fP$ and the Lusztig involution to the marked states of $\overline{\states}_{\lambda,s_1 s_2}$, we obtain the set-valued tableaux
\[
%
%
\begin{array}{c@{\quad}c@{\quad}c@{\quad}c@{\quad}c}
\ytableaushort{1111,2{2,\!3},3}\,,
&
\ytableaushort{111{1,\!2},2{2,\!3},3}\,,
&
\ytableaushort{111{2,\!3},22,3}\,,
&
\ytableaushort{11{1,\!2}2,2{2,\!3},3}\,,
&
\ytableaushort{112{2,\!3},22,3}\,,
\\[30pt]
\ytableaushort{1112,23,3}\,,
&
\ytableaushort{1112,2{2,\!3},3}\,,
&
\ytableaushort{1122,23,3}\,,
&
\ytableaushort{1122,2{2,\!3},3}\,,
&
\ytableaushort{1222,23,3}\,.
\end{array}
\]
These are precisely the elements of the set $\{T \in \svt^n(\lambda) \mid K(T) = K_{w\lambda} \}$, where $K_{w\lambda}$ is the lower right semistandard Young tableau given above.
Taking the sum over these set-valued tableaux, we obtain
\[
\beta^2 (z_1^4 z_2^3 z_3^2 + z_1^3 z_2^4 z_3^2)
+ \beta (z_1^4 z_2^2 z_3^2 + 2 z_1^3 z_2^3 z_3^2 + 2 z_1^2 z_2^4 z_3^2)
+ z_1^3 z_2^2 z_3^2  + z_1^2 z_2^3 z_3^2 + z_1 z_2^4 z_3^2,
\]
which equals $\overline{L}_{s_1 s_2 \lambda}(z_1, z_2, z_3; \beta)$ as stated by Theorem~\ref{thm:key_theorem}.
\end{example}

\begin{theorem}
\label{thm:skyline_bijection}
Conjecture~\ref{conj:skyline_tableaux} is true, that is to say
\[
\overline{L}_{w\lambda} = \sum_{S \in \skyline(w\lambda)} \beta^{\excess(S)} \zz^{\wt(S)}.
\]
\end{theorem}

\begin{proof}
Recall the bijection $\phi$ between marked GT patterns and set-valued tableaux.
As previously mentioned, the map $\phi \circ \fP$ between marked states and set-valued tableaux is not weight preserving, but if we instead consider \emph{reverse} set-valued tableaux by replacing each $i \mapsto n+1-i$, then this modified bijection $\psi \colon \bigsqcup_{w \in \sym{n}} \overline{\states}_{\lambda,w} \to \rsvt^n(\lambda)$ is weight preserving.
Furthermore, we note that for a marked state $(S, M)$ such that $\psi(S, M) = T$, we have $\psi(S, \emptyset) = \max T$.
Next, we restrict the domain of $\psi$ to the unmarked states in $\overline{\states}_{\lambda,w}$, then we obtain a set of reverse semistandard tableaux that is in bijection with semistandard skyline tableaux by~\cite{Mason08}.
We denote this bijection by $\eta$, which is given by organizing the columns so that the skyline tableaux conditions are satisfied.
The possible markings of a particular state $S$ correspond to the boxes in $\psi(S, \emptyset)$ where we can add \emph{smaller} entries and still have a reversed set-valued tableau.
Next, the bijection $\eta$ is extended to reversed set-valued tableaux and set-valued skyline tableaux in~\cite[Thm.~2.4]{Monical16} by simply placing the free entries in each column in the appropriate locations.
Therefore, the map $\eta \circ \psi|_{\overline{\states}_{\lambda,w}}$ is a weight preserving bijection from the marked states $\overline{\states}_{\lambda,w}$ to $\skyline(w\lambda)$ as any possible corner to be marked in a state corresponds to a possible free entry in the set-valued skyline tableau.
Hence, the claim follows.
\end{proof}

\begin{example}
Consider the state $S$ from Figure~\ref{fig:coloredstate}.
The two monomials in the Boltzmann weight $\wt(S) = z_1^3 z_2^2 z_3^2 + \beta z_1^4 z_2^2 z_3^2$ correspond to, under $\psi$, the reverse set-valued tableaux
\[
\ytableaushort{3311,22,1}\,,
\qquad\qquad
\ytableaushort{3311,2{2,\!1},1}\,,
\]
with the one on the right corresponding to marking the unique $\tt{a}_2$ vertex in $S$.
By then applying $\eta$, we obtain the set-valued skyline tableaux
\[
\ytableaushort{1,2211,33}\,,
\qquad\qquad
\ytableaushort{1,2{2,\!1}11,33}\,.
\]
\end{example}

\appendix
\section{\texorpdfstring{\textsc{SageMath}}{SageMath} code for computing the \texorpdfstring{$R$}{R}-matrix}
\label{sec:sage_code}

We give the \textsc{SageMath}~\cite{sage} code we used to compute the $R$-matrix such that Proposition~\ref{prop:YBE} holds.

\lstset{numbers=left}
\begin{lstlisting}
sage: A.<b,z1,z2> = ZZ[]
sage: B.<x1,x2,x3,x4,x5,x6,x7,x8,x9,x10,x11,x12,x13,x14> = A[]
sage: def L_wt(u, r, d, l, z, crosses=[]):
....:     if u == r == d == l == 0:
....:         return 1
....:     if u == l and d == r:
....:         if u == 0:
....:             return 1+b*z
....:         if r == 0:
....:             return 1
....:         if ((u,r) in crosses or (r,u) in crosses):
....:             if u <= r:
....:                 return 1
....:         else:
....:             if u >= r:
....:                 return 1
....:         return 0
....:     if l == r and u == d:
....:         if u == 0:
....:             return z
....:         if u > r > 0 and ((u,r) in crosses or (r,u) in crosses):
....:             return 1
....:         return 0
....:     return 0
sage: def R_wt(ur, lr, ll, ul, crosses=[]):
....:     if ur == lr == ll == ul == 0:
....:         return x1
....:     if ur == lr == ll == ul:
....:         return x8
....:     if ul == ur and ll == lr:
....:         if ll == 0:
....:             return x2
....:         if ul == 0:
....:             return x9
....:         if ((ul,ll) in crossings or (ll,ul) in crosses):
....:             if ul > ll:
....:                 return x5
....:             if ul < ll:
....:                 return x6
....:         if ((ul,ll) in crossings or (ll,ul) in crosses):
....:             if ul > ll:
....:                 return x11
....:             if ul < ll:
....:                 return x12
....:         raise AssertionError("the all equal case done previously")
....:     if ul == lr and ll == ur:
....:         if ll == 0:
....:             return x3
....:         if ul == 0:
....:             return x4
....:         if ((ul,ll) in crossings or (ll,ul) in crossings):
....:             if ul < ll:
....:                 return x7
....:              if ul > ll:
....:                 return x10
....:         if ((ul,ll) in crossings or (ll,ul) in crossings):
....:             if ul < ll:
....:                 return x13
....:              if ul > ll:
....:                 return x14
....:         raise AssertionError("the all equal case done previously")
....:     return 0
sage: states = list(cartesian_product([[0,1,2,3],[0,1,2,3],[0,1,2,3]]))
sage: R = matrix([[R_wt(s[1],s[0],t[1],t[0]) if s[2] == t[2] else 0
....:              for t in states] for s in states])
sage: L1 = matrix([[L_wt(s[2],s[0],t[2],t[0],z2) if s[1] == t[1] else 0
....:               for t in states] for s in states])
sage: L2 = matrix([[L_wt(s[2],s[1],t[2],t[1],z1) if s[0] == t[0] else 0
....:               for t in states] for s in states])
sage: RLL = R*L1*L2 - L2*L1*R
sage: RLLSR = RLL.change_ring(SR)
sage: solve([RLLSR[i,j] == 0 for i in range(len(states))
....:        for j in range(len(states))],
....:       SR.var('x1,x2,x3,x4,x5,x6,x7,x8,x9,x10,x11,x12,x13,x14'))
[[x1 == (b*r1*z1 + r1)/(b*z2 + 1), x2 == (b*r1*z1 + r1)/(b*z2 + 1),
  x3 == -(r1*z1 - r1*z2)/(b*z2 + 1), x4 == 0, x5 == r1*z1/z2,
  x6 == (b*r1*z1 + r1)/(b*z2 + 1), x7 == -(r1*z1 - r1*z2)/(b*z2^2 + z2),
  x8 == (b*r1*z1 + r1)/(b*z2 + 1), x9 == r1, x10 == 0]]
sage: Rp = R(x1=(1+b*z1)*z2, x2=(1+b*z1)*z2, x3=(z2-z1)*z2, x4=0,
....:        x5=(1+b*z2)*z1, x6=(1+b*z1)*z2, x7=z2-z1, x8=(1+b*z1)*z2,
....:        x9=(1+b*z2)*z2, x10=0, x11 == (b*r10*z1 + r10)/(b*z2 + 1),
....:        x12 == r10*z1/z2, x13 == 0, x14 == 0)
sage: Rp*L1*L2 == L2*L1*Rp
True
\end{lstlisting}
To show Proposition~\ref{prop:uncolored_YBE}, instead use
\lstset{numbers=none}
\begin{lstlisting}
sage: states = list(cartesian_product([[0,1],[0,1],[0,1]]))
\end{lstlisting}
and ignore the variables $x_5, x_6, x_7, x_{10}, x_{11}, x_{12}, x_{13}, x_{14}$.
We can compute the $R$-matrix for Proposition~\ref{prop:Lascoux_prime_YBE} by changing Line~12 above to
\begin{lstlisting}
....:             if u <= r:
\end{lstlisting}

\bibliographystyle{alpha} 
\bibliography{grothendiecks}

\end{document}